\documentclass[12pt,leqno]{article}

\usepackage{amsmath,amssymb,amsthm,amscd}
\usepackage[all]{xy}

\usepackage[top=30truemm,bottom=30truemm,left=25truemm,right=25truemm]{geometry}

\makeatletter
    
    \@addtoreset{equation}{section}
  \makeatother

\theoremstyle{plain}
\newtheorem{thm}{Theorem}[section]
\newtheorem{prop}[thm]{Proposition}
\newtheorem{lem}[thm]{Lemma}
\newtheorem{cor}[thm]{Corollary}

\theoremstyle{definition}
\newtheorem{defn}[thm]{Definition}

\renewcommand{\labelenumi}{(\roman{enumi})}

\newcommand{\FRAC}[2]{\leavevmode\kern.1em\raise.5ex\hbox{\the\scriptfont0 #1}\kern-.1em/\kern-.15em\lower.25ex\hbox{\the\scriptfont0 #2}}

\newcommand{\BQ}{{\mathbf{Q}}}

\newcommand{\BZ}{{\mathbf{Z}}}

\newcommand{\dvr}{\mathcal{O}}

\newcommand{\et}{\mathrm{\acute{e}t}}

\newcommand{\hok}{H^{1}(K,\mathbf{Q}/\mathbf{Z})}

\newcommand{\sw}{\mathrm{sw}}
\newcommand{\ab}{\mathrm{ab}}
\newcommand{\mf}{\mathcal{F}}

\DeclareMathOperator{\Gal}{Gal}

\newcommand{\pr}{\mathrm{pr}}

\newcommand{\rsw}{\mathrm{rsw}}

\DeclareMathOperator{\cform}{char}

\newcommand{\fillog}{\mathrm{fil}}
\newcommand{\fil}{\fillog'}
\DeclareMathOperator{\art}{dt}
\newcommand{\wsk}{W_{s}(K)}
\DeclareMathOperator{\ord}{ord}
\newcommand{\oko}{\Omega_{K}^{1}}
\newcommand{\gr}{\mathrm{gr}'}
\newcommand{\grlog}{\mathrm{gr}}

\DeclareMathOperator{\Supp}{Supp}
\DeclareMathOperator{\Char}{Char}

\newcommand{\dt}{\mathrm{dt}}

\newcommand{\mg}{\mathcal{G}}

\newcommand{\cdvr}{\hat{\dvr}}
\newcommand{\mh}{\mathcal{H}}

\DeclareMathOperator{\Spec}{Spec}
\DeclareMathOperator{\Frac}{Frac}

\DeclareMathOperator{\Ker}{Ker}

\begin{document}
\title
{Equality of two non-logarithmic ramification filtrations
of abelianized Galois group in positive characteristic}
\author{YURI YATAGAWA}
\date{}
\maketitle
\begin{abstract}
We prove the equality of two non-logarithmic ramification filtrations defined by Matsuda
and Abbes-Saito for the abelianized absolute Galois group of a complete discrete valuation field
in positive characteristic.

We compute the refined Swan conductor and the characteristic form of a character of the fundamental 
group of a smooth separated scheme over a perfect field of positive characteristic
by using sheaves of Witt vectors.
\end{abstract}
\section*{Introduction}
Let $K$ be a complete discrete valuation field with residue field $F_{K}$
and $G_{K}=\Gal(K^{\mathrm{sep}}/K)$ the absolute Galois group of $K$.
In \cite{se}, the definition of (upper numbering) ramification filtration of 
$G_{K}$ is given in the case where $F_{K}$ is perfect.
In the general residue field case, Abbes-Saito (\cite{as1}) have given definitions of two ramification
filtrations of $G_{K}$ geometrically, one is logarithmic and the other is non-logarithmic.
In Saito's recent work (\cite{sa1}, \cite{sa2}) on characteristic cycle of constructible sheaves,
the non-logarithmic filtration in equal characteristic 
plays important roles to give an example of characteristic cycle.

Assume that $K$ is of positive characteristic.
Let $H^{1}(K,\BQ/\BZ)$ be the character group of $G_{K}$.
In this case, Matsuda (\cite{ma}) has defined a non-logarithmic ramification filtration of 
$H^{1}(K,\BQ/\BZ)$ as a non-logarithmic variant of Brylinski-Kato's logarithmic filtration 
(\cite{br}, \cite{ka1}) using Witt vectors.
In this paper, we prove that the abilianization of Abbes-Saito's non-logarithmic filtration 
$\{G_{K}^{r}\}_{r\in \BQ_{\ge 1}}$ is the same as Matsuda's filtration 
$\{\fil_{m}H^{1}(K,\BQ/\BZ)\}_{m\in \BZ_{\ge 1}}$
by taking dual, 
which enable us to compute abelianized Abbes-Saito's filtration by using Witt vectors.
This is stated as follows and proved in Section \ref{seceqfil}:

\begin{thm}
\label{mainthm}
Let $m\ge 1$ be an integer and $r$ a rational number such that $m\le r < m+1$. 
For $\chi\in H^{1}(K, \BQ/\BZ)$,
the following are equivalent:
\begin{enumerate}
\item $\chi \in \fil_{m}H^{1}(K,\mathbf{Q}/\mathbf{Z})$.
\item $\chi(G_{K}^{m+})=0$.
\item $\chi(G_{K}^{r+})=0$.
\end{enumerate}
\end{thm}

For $m>2$, Theorem \ref{mainthm} has been proved by Abbes-Saito (\cite{as3}).
The proof goes similarly as the proof by Abbes-Saito (loc.\ cit.).
The proof in this paper relies on the 
characteristic form defined by Saito (\cite{sa1}) even in the exceptional case where $p=2$
and an explicit computation of the characteristic form. 

Let $X$ be a smooth separated scheme over a perfect field of positive characteristic
and $U=X-D$ the complement of a divisor $D$ on $X$ with simple normal crossings.
The characteristic form of a character of the abelianized fundamental group $\pi_{1}^{\ab}(U)$
is an element of the restriction over a radicial covering of a sub divisor $Z$ of $D$
of differential module of $X$.
We compute the characteristic form using sheaves of Witt vectors.
By taking $X$ and $D$ so that the local field at a generic point of $D$ is $K$
and using the injections defined by the characteristic form
from the graded quotients of $\{\fil_{m}H^{1}(K,\BQ/\BZ)\}_{m\in \BZ_{\ge 1}}$
and the modules of characters of the graded quotients of $\{G_{K}^{r}\}_{r\in \BQ_{\ge 1}}$,
we obtain the proof of Theorem \ref{mainthm}.

This paper consists of three sections.
In Section \ref{seckmram}, we recall Kato and Matsuda's ramification theories in positive characteristic.
We give some complements to these theories to compute the refined Swan conductor (\cite{ka1})
and the characteristic form for a character of the fundamental group of a smooth separated scheme
over a perfect field of positive characteristic in terms of 
sheaves of Witt vectors.
In Section \ref{secasram}, we recall Abbes-Saito's non-logarithmic ramification theory in 
positive characteristic in terms of schemes over a perfect field.
We recall the definition of the characteristic form defined by Saito
and show that this characteristic form is computed with sheaves of Witt vectors. 
Section \ref{seceqfil} is devoted to prove Theorem \ref{mainthm}.

This paper is a refinement of a part of the author's thesis at University of Tokyo.
The author would like to express her sincere gratitude to her supervisor Takeshi Saito 
for suggesting her to refine the construction of characteristic form using sheaves of Witt vectors, 
reading the manuscript carefully, and giving a lot of advice on the manuscript.
The research was partially supported by the Program for Leading Graduate Schools, MEXT, Japan
and JSPS KAKENHI Grant Number 15J03851.

\tableofcontents
\section{Kato and Matsuda's ramification theories and complements}
\label{seckmram}
\subsection{Local theory: logarithmic case}
\label{ssloclog}
We recall Kato's ramification theory (\cite{ka1}, \cite{ka2})
and prove some properties of graded quotients 
of some filtrations
for the proof of Proposition \ref{lemsone} in
Subsection \ref{sssheaf}.

Let $K$ be a complete discrete valuation field of characteristic $p>0$.
We regard $H^{1}_{\et}(K,\mathbf{Z}/n\mathbf{Z})$ as a subgroup of $H^{1}_{\et}(K,\mathbf{Q}/\mathbf{Z})=\varinjlim_{n} H^{1}_{\et}(K,\mathbf{Z}/n\mathbf{Z})$.
Let $W_{s}(K)$ be the Witt ring of $K$ of length $s \geq 0$.
By definition, $W_{0}(K)=0$ and $W_{1}(K)=K$.
We write
\begin{equation}
F : W_{s}(K) \rightarrow W_{s}(K) ;\ 
(a_{s-1}, \cdots , a_{0}) \mapsto (a_{s-1}^{p}, \cdots , a_{0}^{p}) \notag
\end{equation}
for the Frobenius.
Since $W_{s}(\mathbf{F}_{p})\simeq \BZ/p^{s}\BZ$, the exact sequence
\begin{equation}
0\rightarrow W_{s}(\mathbf{F}_{p}) \rightarrow W_{s}(K) 
\xrightarrow{F-1} W_{s}(K) \rightarrow 0 \notag 
\end{equation}
induces the exact sequence 
\begin{equation}
\label{exseqdelta}
0 \rightarrow W_{s}(\mathbf{F}_{p}) \rightarrow W_{s}(K) 
\xrightarrow{F-1} W_{s}(K)  
\rightarrow H^{1}(K, \BZ/p^{s}\BZ) \rightarrow 0. 
\end{equation}
We define
\begin{equation}
\label{delts}
\delta _{s}: W_{s}(K) \rightarrow H^{1}(K,\mathbf{Q}/\mathbf{Z}) 
\end{equation}
to be the composition
\begin{equation}
W_{s}(K) \rightarrow H^{1}(K,\BZ/p^{s}\BZ) \rightarrow
H^{1}(K,\mathbf{Q}/\mathbf{Z}), \notag
\end{equation}
where the first arrow is the forth morphism in (\ref{exseqdelta}). 
 
Let $\mathcal{O}_{K}$ be the valuation ring of $K$
and $F_{K}$ the residue field of $K$.
We write $G_{K}$ for the absolute Galois group of $K$.

\begin{defn}[{\cite[Definition (3.1)]{ka1}}]
\label{filw}
Let $s\ge 0$ be an integer.
\begin{enumerate}
\item Let $a=(a_{s-1},\ldots,a_{0})$ be an element of $W_{s}(K)$.
We define $\ord_{K}(a)$ by $\ord_{K}(a)=\min_{0\le i \le s-1}\{p^{i}\ord_{K}(a_{i})\}$.
\item We define an increasing filtration $\{ \fillog_{n}W_{s}(K) \}_{n \in \mathbf{Z}}$
of $W_{s}(K)$ by
\begin{equation}
\label{filwdef}
\fillog_{n}W_{s}(K) = \{ a\in W_{s}(K) \ |\ \ord_{K}(a) 
\geq -n\}.
\end{equation}
\end{enumerate}
\end{defn}

The filtration $\{\fillog_{n}W_{s}(K)\}_{n\in \mathbf{Z}}$ in Definition \ref{filw} is first defined by Brylinski (\cite[Proposition 1]{br})
and $\fillog_{n}W_{s}(K)$ is a submodule of $W_{s}(K)$ for $n\in \BZ$
(loc.\ cit.).

Let $n\ge 0$ be an integer and put $s'=\ord_{p}(n)$.
Suppose that $s'<s$.
Let $V$ denote the Verschiebung
\begin{equation}
V : W_{s}(K) \rightarrow W_{s+1}(K) ;\ (a_{s-1}, \cdots, a_{0}) \mapsto (0, a_{s-1}, \cdots, a_{0}). \notag
\end{equation}
Since $(a_{s-1},\ldots,a_{0})=(a_{s-1},\ldots,a_{s'+1},0,\ldots,0)+V^{s-s'-1}(a_{s'},\ldots,a_{0})$, 
we have
\begin{equation}
\label{fillogw}
\fillog_{n}W_{s}(K)=\fillog_{n-1}W_{s}(K)+V^{s-s^{\prime}-1}\fillog_{n}W_{s^{\prime}+1}(K).
\end{equation}

\begin{defn}[{\cite[Corollary (2.5), Theorem (3.2) (1)]{ka1}}]
\label{defoffiltofh1}
Let $\delta_{s}$ be as in (\ref{delts}).
\begin{enumerate}
\item We define an increasing filtration $\{\fillog_{n}H^{1}(K,\BZ/p^{s}\BZ)\}_{n\in \BZ_{\ge 0}}$
of $H^{1}(K,\BZ/p^{s}\BZ)$ by 
\begin{equation}
\fillog_{n}H^{1}(K,\BZ/p^{s}\BZ)=\delta_{s}(\fillog_{n}W_{s}(K)). \notag
\end{equation}
\item We define an increasing filtration 
$\{ \fillog_{n}H^{1}(K,\mathbf{Q}/\mathbf{Z}) \}_{n \in \mathbf{Z}_{\geq 0}}$ 
of $H^{1}(K,\mathbf{Q}/\mathbf{Z})$ by
\begin{equation}
\fillog_{n}H^{1}(K,\mathbf{Q}/\mathbf{Z}) = H^{1}(K,\mathbf{Q}/\mathbf{Z})\{p^{\prime}\} + \bigcup_{s \geq 1} \delta_{s}(\fillog_{n}W_{s}(K)),
\end{equation}
where $H^{1}(K,\mathbf{Q}/\mathbf{Z})\{p^{\prime}\}$ denotes the prime-to-$p$ part of $H^{1}(K,\mathbf{Q}/\mathbf{Z})$.
\end{enumerate}
\end{defn}

\begin{defn}[{\cite[Definition (2.2)]{ka1}}]
\label{dsa}
Let $\chi$ be an element of $H^{1}(K,\mathbf{Q}/\mathbf{Z})$.
We define the \textit{Swan conductor} $\sw(\chi)$ of $\chi$ by
$\mathrm{sw}(\chi) = \min \{ n\in \mathbf{Z}_{\geq 0}\ |\ \chi \in \fillog_{n}H^{1}(K,\mathbf{Q}/\mathbf{Z})\}$.
\end{defn}

We recall the definition of refined Swan conductor of $\chi \in \hok$
given by Kato (\cite[(3.4.2)]{ka2}).
Let $\Omega_{K}^{1}$ be the differential module of $K$ over $K^{p}\subset K$.

\begin{defn}
\label{defoffilomega}
We define an increasing filtration $\{ \fillog_{n}\Omega_{K}^{1}\}_{n\in \mathbf{Z}_{\geq 0}}$ of $\Omega^{1}_{K}$ by
\begin{equation}
\fillog_{n}\Omega^{1}_{K} = \{ (\alpha d \pi/\pi+\beta)/\pi^{n} \ |\ \alpha \in \mathcal{O}_{K}, \beta \in \Omega_{\mathcal{O}_{K}}^{1} \}
=\mathfrak{m}^{-n}\Omega_{\dvr_{K}}^{1}(\log ),
\end{equation}
where $\pi$ is a uniformizer of $K$ and $\mathfrak{m}$ is the maximal ideal of $\dvr_{K}$.
\end{defn}

We consider the morphism
\begin{equation}
\label{fsd}
-F^{s-1}d : \wsk \rightarrow \Omega_{K}^{1} ;\ (a_{s-1}, \cdots , a_{0}) \mapsto 
-\sum_{i=0}^{s-1}a_{i}^{p^{i}-1}da_{i}. 
\end{equation}
The morphism $-F^{s-1}d$ (\ref{fsd}) satisfies 
$-F^{s-1}d(\fillog_{n}\wsk)\subset \fillog_{n}\Omega_{K}^{1}$.
We put $\grlog_{n} = \fillog_{n}/\fillog_{n-1}$ for $n\in \BZ_{\geq 1}$.
Then, for $n \in \BZ_{\ge 1}$, the morphism (\ref{fsd}) induces
\begin{equation}
{\varphi_{s}}^{(n)} \colon \grlog_{n}\wsk \rightarrow \grlog_{n}\Omega_{K}^{1}. \notag
\end{equation}

Let $\delta_{s}^{(n)}\colon \grlog_{n}W_{s}(K)\rightarrow \grlog_{n}H^{1}(K,\mathbf{Q}/\mathbf{Z})$ 
denote the morphism induced by $\delta_{s}$ (\ref{delts}) for $n\in \mathbf{Z}_{\ge 1}$.
For $n\in\mathbf{Z}_{\ge 1}$, 
there exists a unique injection $\phi^{(n)}\colon \grlog_{n}\hok \rightarrow \grlog_{n}\oko$ such that the diagram
\begin{equation}
\label{refinedswaninj}
\xymatrix{\grlog_{n}\wsk \ar[rr]^{\varphi_{s}^{(n)}} \ar[rd]_{\delta_{s}^{(n)}} & & \grlog_{n}\Omega^{1}_{K} \\
& \grlog_{n}\hok \ar[ru]_{\phi^{(n)}} &
}
\end{equation}
is commutative for any $s\in \mathbf{Z}_{\ge 0}$ by \cite[Remark 3.2.12]{ma}, or \cite[\S 10]{as3} for more detail.
We note that $\grlog_{n}\Omega^{1}_{K}\simeq \mathfrak{m}^{-n}\Omega^{1}_{\dvr_{K}}(\log )\otimes_{\dvr_{K}} F_{K}$ is a vector space over $F_{K}$.

\begin{defn}[{\cite[(3.4.2)]{ka2}}, {\cite[Remark 3.2.12]{ma}}, 
see also {\cite[D\'{e}finition 10.16]{as3}}]
\label{defoflocalrefinedswan}
Let $\chi$ be an element of $\hok$. 
We put $n=\sw(\chi)$.
If $n\ge 1$, then we define the \textit{refined Swan conductor} $\rsw (\chi)$ of $\chi$
to be the image of $\chi$ by $\phi^{(n)}$ in (\ref{refinedswaninj}).
\end{defn} 

In the rest of this subsection, we prove some properties of graded quotients of filtrations.

For $q\in \mathbf{R}$, let $[q]$ denote the integer $n$ such that
$q-1 < n \le q$.

\begin{lem}
\label{lemgaus}
Let $m$ and $r\ge 0$ be integers.
\begin{enumerate}
\item $[m/p^{r}]=[(m-1)/p^{r}]+1$ if $m\in p^{r}\BZ$ and
$[m/p^{r}]=[(m-1)/p^{r}]$ if $m\notin p^{r}\BZ$.
\item $[[m/p^{r}]/p]=[m/p^{r+1}]=[[m/p]/p^{r}]$.
\end{enumerate}
\end{lem}

\begin{proof}
(i) We put $m=p^{r}q+a$, where $q,a\in \BZ$ and $0\le a<p^{r}$.
Then $[m/p^{r}]=q$.
Further $[(m-1)/p^{r}]=q+[(a-1)/p^{r}]$.
Since $[(a-1)/p^{r}]=-1$ if $a=0$ and $[(a-1)/p^{r}]=0$ if $0<a<p^{r}$,
the assertion follows.

(ii) We put $m=p^{r+1}q'+a'$, where $q',a'\in \BZ$ and $0\le a'<p^{r+1}$.
Then $[m/p^{r}]=pq'+[a'/p^{r}]$ and $0\le [a'/p^{r}] <p$.
Further $[m/p]=p^{r}q'+[a'/p]$ and $0\le[a'/p]<p^{r}$.
Hence we have $[[m/p^{r}]/p]=q'=[m/p^{r+1}]$ and $[[m/p]/p^{r}]=q'=[m/p^{r+1}]$.
\end{proof}

\begin{lem}
\label{lempbbyf}
Let $a$ be an element of $W_{s}(K)$.
\begin{enumerate}
\item $\ord_{K}(F(a))=p\cdot\ord_{K}(a)$.
\item $\ord_{K}((F-1)(a))=p\cdot\ord_{K}(a)$ if $\ord_{K}(a)<0$ 
and $\ord_{K}((F-1)(a))\ge 0$ if $\ord_{K}(a)\ge 0$.
\item For an integer $n\ge 0$, we have
$F^{-1}(\fillog_{n}W_{s}(K))
=(F-1)^{-1}(\fillog_{n}W_{s}(K))
=\fillog_{[n/p]}W_{s}(K)$.
\end{enumerate}
\end{lem}

\begin{proof}
(i) We put  $a=(a_{s-1},\ldots,a_{0})$.
Since $F(a)=(a_{s-1}^{p},\ldots,a_{0}^{p})$,
the assertion follows.

(ii) Suppose that $\ord_{K}(a)\ge 0$.
Then, since both $a$ and $F(a)$ belong to $\fillog_{0}W_{s}(K)$, we have $(F-1)(a)\in \fillog_{0}W_{s}(K)$.
Hence we have $\ord_{K}((F-1)(a))\ge 0$ by (\ref{filwdef}).

Suppose that $\ord_{K}(a)<0$.
We put $\ord_{K}(a)=-n$.
Since both $a$ and $F(a)$ belong to $\fillog_{pn}W_{s}(K)$, we have $(F-1)(a)\in \fillog_{pn}W_{s}(K)$.
Since $\ord_{K}(F(a))=-pn< \ord_{K}(a)=-n$, 
we have $(F-1)(a)\notin \fillog_{pn-1}W_{s}(K)$.
Hence we have $\ord_{K}((F-1)(a))=-pn$.

(iii) By (i), we have $F(a)\in \fillog_{n}W_{s}(K)$ if and only if
$\ord_{K}(a)\ge -n/p$ for $a\in W_{s}(K)$.
Hence we have $F^{-1}(\fillog_{n}W_{s}(K))=\fillog_{[n/p]}W_{s}(K)$.
By (ii), we have $(F-1)^{-1}(\fillog_{n}W_{s}(K))=\fillog_{[n/p]}W_{s}(K)$
similarly.
\end{proof}

Let $n\ge 1$ be an integer.
By Lemma \ref{lempbbyf} (iii), the Frobenius $F\colon W_{s}(K)\rightarrow
W_{s}(K)$ induces the injection
\begin{equation}
\label{barfk}
\bar{F}\colon \fillog_{[n/p]}W_{s}(K)/\fillog_{[(n-1)/p]}W_{s}(K) 
\rightarrow \grlog_{n}W_{s}(K).
\end{equation}
By Lemma \ref{lemgaus} (i), the domain of (\ref{barfk}) is equal to
$\grlog_{n/p}W_{s}(K)$ if $n\in p\BZ$ and it is $0$ if $n\notin p\BZ$.

By Lemma \ref{lempbbyf} (iii), the morphism $F-1\colon W_{s}(K)\rightarrow W_{s}(K)$ 
induces the injection
\begin{equation}
\label{barfmok}
\overline{F-1}\colon \fillog_{[n/p]}W_{s}(K)/\fillog_{[(n-1)/p]}W_{s}(K) 
\rightarrow \grlog_{n}W_{s}(K). 
\end{equation}
Since $[n/p]< n$ if $n\ge 1$,
the morphisms (\ref{barfk}) and (\ref{barfmok}) are the same.

\begin{lem}[{cf.\ \cite[Theorem (3.2), Corollary (3.3)]{ka1}}]
\label{exseqgrwk}
Let $n\ge 1$ be an integer.
Then we have the exact sequence
\begin{equation}
0\rightarrow \fillog_{[n/p]}W_{s}(K)/\fillog_{[(n-1)/p]}W_{s}(K)\xrightarrow{\bar{F}} \grlog_{n}W_{s}(K)
\xrightarrow{\varphi^{(n)}_{s}}\grlog_{n}\Omega^{1}_{K},
\notag
\end{equation}
where $\fillog_{[n/p]}W_{s}(K)/\fillog_{[(n-1)/p]}W_{s}(K)$ is $\grlog_{n/p}W_{s}(K)$ if $n\in p\BZ$ and
$0$ if $n\notin p\BZ$.
\end{lem}

\begin{proof}
As in the proof of \cite[Proposition 10.7]{as3}, the morphism $\varphi_{s}^{(n)}$ factors through 
\begin{equation}
\grlog_{n}H^{1}(K,\BZ/p^{s}\BZ)\simeq \fillog_{n}W_{s}(K)/((F-1)(W_{s}(K))\cap \fillog_{n}W_{s}(K)+\fillog_{n-1}W_{s}(K)). \notag
\end{equation}
Since this factorization defines the injection $\phi^{(n)}$ in (\ref{refinedswaninj})
by \cite[Proposition 10.14]{as3}
and since the morphism $\bar{F}$ (\ref{barfk}) is equal to 
the morphism $\overline{F-1}$ (\ref{barfmok}),
the assertion follows.
\end{proof}

\begin{defn}
Let $s\ge 0$ and $r\ge 0$ be integers.
We define an increasing filtration $\{\fillog^{(r)}_{n}W_{s}(K)\}_{n\in \mathbf{Z}_{\ge 0}}$ of $W_{s}(K)$ by
\begin{equation}
\label{filrdef}
\fillog_{n}^{(r)}W_{s}(K)=\{a\in W_{s}(K)\; |\;
\ord_{K}(a)\ge -n/p^{r}\}=\fillog_{[n/p^{r}]}W_{s}(K).
\end{equation}
\end{defn}

By (\ref{filrdef}), we have $\fillog_{n}^{(0)}W_{s}(K)=\fillog_{n}W_{s}(K)$ for $n\in \BZ_{\ge 0}$.

For integers $0\le t\le s$,
let $\pr_{t}$ denote the projection
\begin{equation}
\label{prt}
\pr_{t}\colon W_{s}(K)\rightarrow W_{t}(K)
\; ; \; (a_{s-1},\ldots,a_{0})\mapsto (a_{s-1},\ldots,a_{s-t}).
\end{equation}
We put $\grlog_{n}^{(r)}=\fillog_{n}^{(r)}/\fillog_{n-1}^{(r)}$ for $r\in \BZ_{\ge 0}$ and $n\in \BZ_{\ge 1}$.

\begin{lem}
\label{lemfilrep}
Let $r\ge 0$ and $0\le t\le s$ be integers.
Let $\pr_{t}\colon W_{s}(K)\rightarrow W_{t}(K)$ be as in (\ref{prt}).
Let $n\ge 0$ be an integer.
\begin{enumerate}
\item $\pr_{t}(\fillog_{n}W_{s}(K))=\fillog^{(s-t)}_{n}W_{t}(K)$.
\item $(F-1)^{-1}(\fillog_{n}^{(r)}W_{s}(K))=\fillog_{[n/p]}^{(r)}W_{s}(K)$.
\end{enumerate}
\end{lem}

\begin{proof}
(i) By (\ref{filwdef}), we have $\pr_{t}(\fillog_{n}W_{s}(K))=\fillog_{[n/p^{s-t}]}W_{t}(K)$.
Hence the assertion follows by (\ref{filrdef}).

(ii) By Lemma \ref{lempbbyf} (iii) and (\ref{filrdef}), we have $(F-1)^{-1}(\fillog_{n}^{(r)}W_{s}(K))=
\fillog_{[[n/p^{r}]/p]}W_{s}(K)$.
By Lemma \ref{lemgaus} (ii) and (\ref{filrdef}), the assertion follows.
\end{proof}

Let $n\ge 0$ and $0\le t\le s$ be integers. 
Since $\pr_{t}(\fillog_{n}W_{s}(K))=\fillog^{(s-t)}_{n}W_{t}(K)$
by Lemma \ref{lemfilrep} (i), we have the exact sequence
\begin{equation}
\label{esprvl}
0\rightarrow \fillog_{n}W_{s-t}(K) \xrightarrow{V^{t}} 
\fillog_{n}W_{s}(K) \xrightarrow{\pr_{t}} 
\fillog_{n}^{(s-t)}W_{t}(K) \rightarrow 0.
\end{equation}

\begin{lem}
\label{lemesgrk}
Let $n\ge 1$ be an integer.
Then the exact sequence (\ref{esprvl}) induces the exact sequence
\begin{equation}
0\rightarrow \grlog_{n}W_{s-t}(K)\xrightarrow{\bar{V}^{t}} \grlog_{n}W_{s}(K)
\xrightarrow{\overline{\pr}_{t}} \grlog_{n}^{(s-t)}W_{t}(K) \rightarrow 0,
\notag
\end{equation}
where $\grlog_{n}^{(s-t)}W_{t}(K)$ is equal to $\grlog_{n/p^{s-t}}W_{t}(K)$ if $n\in p^{s-t}\BZ$
and $0$ if $n\notin p^{s-t}\BZ$.
\end{lem}

\begin{proof}
We consider the commutative diagram
\begin{equation}
\label{comeqseq}
\xymatrix{
0\ar[r] & \fillog_{n-1}W_{s-t}(K) \ar[r]^-{V^{t}} \ar[d] &
\fillog_{n-1}W_{s}(K) \ar[r]^-{\pr_{t}} \ar[d] &
\fillog_{n-1}^{(s-t)}W_{t}(K) \ar[r] \ar[d] & 0 \\
0\ar[r] & \fillog_{n}W_{s-t}(K) \ar[r]^-{V^{t}} &
\fillog_{n}W_{s}(K) \ar[r]^-{\pr_{t}} &
\fillog_{n}^{(s-t)}W_{t}(K) \ar[r] & 0,
}
\end{equation}
where the horizontal lines are exact and the vertical arrows are inclusions.
By applying the snake lemma to (\ref{comeqseq}), we obtain the exact
sequence which we have desired.
The last supplement to $\grlog_{n}^{(s-t)}W_{t}(K)$
follows by Lemma \ref{lemgaus} (i) and (\ref{filrdef}).
\end{proof}

\subsection{Local theory: non-logarithmic case}
\label{ssloclnlog}

We recall a non-logarithmic variant, given by Matsuda (\cite{ma}), of Kato's logarithmic 
ramification theory recalled in Subsection \ref{ssloclog},
and we consider the exceptional case of Matsuda's theory.
We also consider the graded quotients of filtrations.
We keep the notation in Subsection \ref{ssloclog}.

\begin{defn}[cf.\ {\cite[3.1]{ma}}]
\label{filpw}
We define an increasing filtration $\{ \fil_{m}W_{s}(K) \}_{m \in\mathbf{Z}_{\geq 1}}$ of $W_{s}(K)$ by
\begin{equation}
\label{filpwdef}
\fil_{m}W_{s}(K) = \fillog_{m-1}W_{s}(K) + 
V^{s-s^{\prime}}\fillog_{m}W_{s^{\prime}}(K).
\end{equation}
Here $s^{\prime}= \mathrm{min} \{\mathrm{ord}_{p}(m), s\}$.
\end{defn}

The definition of $\{ \fil_{m}W_{s}(K) \}_{m \in \mathbf{Z}_{\geq 1}}$ in 
Definition \ref{filpw} is shifted by 1 from Matsuda's definition 
(\cite[3.1]{ma}).
Since $\fillog_{n}W_{s}(K)$ is a submodule of $W_{s}(K)$ for $n\in \BZ$, 
the subset $\fil_{m}W_{s}(K)$ is a submodule of $W_{s}(K)$ for $m\in \BZ_{\ge 1}$.

By (\ref{filpwdef}),
we have
\begin{equation}
\label{seqofwsk}
\fillog_{m-1}\wsk \subset \fil_{m} \wsk \subset \fillog_{m} \wsk 
\end{equation}
for $m\in \mathbf{Z}_{\ge 1}$.
Since $\min \{ \ord_{p}(1), s \} = 0$ for $s \in \BZ_{\geq 0}$, we have 
\begin{equation}
\label{filzwfilpow}
\fillog_{0} W_{s}(K)=\fil_{1}W_{s}(K).
\end{equation} 

\begin{defn}[{cf.\ \cite[Definition 3.1.1]{ma}}]
Let $\delta_{s}$ be as in (\ref{delts}).
\label{deffilpho}
\begin{enumerate}
\item We define an increasing filtration $\{\fil_{m}H^{1}(K,\BZ/p^{s}\BZ)\}_{m\in \BZ_{\ge 1}}$
of $H^{1}(K,\BZ/p^{s}\BZ)$ by 
\begin{equation}
\fil_{m}H^{1}(K,\BZ/p^{s}\BZ)=\delta_{s}(\fil_{m}W_{s}(K)). \notag
\end{equation}
\item We define an increasing filtration $\{ \fil_{m}H^{1}(K,\mathbf{Q}/\mathbf{Z}) \}_{m \in \mathbf{Z}_{\geq 1}}$ of $H^{1}(K,\mathbf{Q}/\mathbf{Z})$ by
\begin{equation}
\fil_{m}H^{1}(K,\mathbf{Q}/\mathbf{Z}) = H^{1}(K,\mathbf{Q}/\mathbf{Z})\{p^{\prime}\} + \bigcup_{s \geq 1} \delta_{s}(\fil_{m}W_{s}(K)),
\end{equation}
where $H^{1}(K,\mathbf{Q}/\mathbf{Z})\{p^{\prime}\}$ denotes the prime-to-$p$ part of $H^{1}(K,\mathbf{Q}/\mathbf{Z})$.
\end{enumerate}
\end{defn}

By (\ref{seqofwsk}),
we have 
\begin{equation}
\label{seqofsethone}
\fillog_{m-1}\hok \subset \fil_{m} \hok \subset \fillog_{m} \hok 
\end{equation}
for $m\in \mathbf{Z}_{\ge 1}$.
By (\ref{filzwfilpow}), we have $\fillog_{0} \hok =\fil_{1} \hok$.

\begin{defn}[cf.\ {\cite[Definition 3.2.5]{ma}}]
\label{dsap}
Let $\chi$ be an element of $H^{1}(K,\mathbf{Q}/\mathbf{Z})$.
We define the \textit{total dimension} $\dt(\chi)$ of $\chi$ by
$\art(\chi) = \min \{ m\in \mathbf{Z}_{\geq 1}\ |\ \chi \in \fil_{m}H^{1}(K,\mathbf{Q}/\mathbf{Z})\}$.
\end{defn}

\begin{defn}
\label{deffilpomega}
We define an increasing filtration $\{ \fil_{m} \Omega_{K}^{1}\}_{m\in \mathbf{Z}_{\geq 1}}$ of $\Omega^{1}_{K}$ by
\begin{equation}
\fil_{m}\Omega^{1}_{K} = \{\gamma/\pi^{m} \ |\ \gamma \in \Omega_{\mathcal{O}_{K}}^{1} \} = \mathfrak{m}^{-m}\Omega_{\mathcal{O}_{K}}^{1},\notag 
\end{equation}
where $\pi$ is a uniformizer of $K$ and $\mathfrak{m}$ is the maximal ideal of $\dvr_{K}$.
\end{defn}

Since $\mathfrak{m}\Omega^{1}_{\dvr_{K}}(\log )\subset \Omega^{1}_{\dvr_{K}}\subset
\Omega^{1}_{\dvr_{K}}(\log)$, we have
\begin{equation}
\label{filomseq}
\fillog_{m-1}\Omega^{1}_{K} \subset \fil_{m} \Omega^{1}_{K} \subset \fillog_{m} \Omega^{1}_{K} 
\end{equation}
for $m\in \mathbf{Z}_{\ge 1}$.

We consider the morphism (\ref{fsd}).
The morphism (\ref{fsd}) satisfies 
$-F^{s-1}d(\fil_{m}\wsk)\subset \fil_{m}\oko$ for $m\in \BZ_{\ge 1}$.
We put $\gr_{m} = \fil_{m}/\fil_{m-1}$ for $m\in \BZ_{\geq 2}$.
Then, for $m\in \BZ_{\ge 2}$, the morphism (\ref{fsd}) induces
\begin{equation}
\label{vphipk}
{\varphi^{\prime}_{s}}^{(m)} \colon \gr_{m}\wsk \rightarrow \gr_{m} \Omega_{K}^{1}. 
\end{equation}

Let $\delta_{s}^{\prime (m)}\colon \gr_{m}W_{s}(K)\rightarrow \gr_{m}H^{1}(K,\mathbf{Q}/\mathbf{Z})$  
denote the morphism induced by $\delta_{s}$ (\ref{delts}) for $m\in \mathbf{Z}_{\ge 2}$.
If $(p,m)\neq (2,2)$, there exists a unique injection 
$\phi^{\prime (m)}\colon \gr_{m} \hok \rightarrow \gr_{m} \oko$ such that the diagram 
\begin{equation}
\label{charforminj}
\xymatrix{
\gr_{m}W_{s}(K) \ar[rr]^{\varphi_{s}^{\prime (m)}} \ar[rd]_{\delta_{s}^{\prime (m)}} & & \gr_{m}\Omega^{1}_{K} \\
&\gr_{m}\hok \ar[ru]_{\phi^{\prime (m)}} & 
}
\end{equation}
is commutative for any $s\in \mathbf{Z}_{\ge 0}$ by \cite[Proposition 3.2.3]{ma}.
We note that $\gr_{m}\Omega^{1}_{K}\simeq \mathfrak{m}^{-m}\Omega^{1}_{\dvr_{K}}\otimes_{\dvr_{K}}F_{K}$ is a vector space over $F_{K}$.

We consider the exceptional case where $(p,m)=(2,2)$.

\begin{lem}
\label{gr1l}
Let $s\ge 1$ be an integer.
Assume that $p=2$.
Then $V^{s-1}\colon K\rightarrow W_{s}(K)$ induces an isomorphism
$\gr_{2}K\rightarrow \gr_{2}W_{s}(K)$.
\end{lem}

\begin{proof}
Since $p=2$, we have $s'=\min\{\ord_{p}(2),s\}=1$.
Hence we have
\begin{align}
\fil_{2}\wsk &= \fillog_{1}\wsk + V^{s-1}\fillog_{2}K \notag\\
&=\fil_{1}W_{s}(K)+V^{s-1}\fillog_{2}K \notag
\end{align} 
by (\ref{filpwdef}) applied at the first equality and by (\ref{fillogw}) and (\ref{filzwfilpow})
applied at the second equality.
Since $\fillog_{2}K=\fil_{2}K$ by (\ref{filpwdef}), the assertion follows.
\end{proof}

\begin{prop}
\label{propnrar}
Assume that $p=2$.
Let $F_{K}^{1/2}\subset \bar{F}_{K}$ denote the subfield of an algebraic 
closure $\bar{F}_{K}$ of $F_{K}$ consisting of the square roots 
of $F_{K}$.
\begin{enumerate}
\item There exists a unique morphism 
\begin{equation}
\tilde{\varphi}^{\prime (2)}_{s} : \gr_{2} \wsk \rightarrow \gr_{2} \oko \otimes_{F_{K}} F_{K}^{1/2} \notag
\end{equation}
such that $\tilde{\varphi}^{\prime (2)}_{s}(\bar{a})=
-da_{0}+\sqrt{\overline{\pi^{2}a_{0}}}d\pi/\pi^{2}$
for every $\bar{a}\in \gr_{2}W_{s}(K)$ whose lift in $\fil_{2}W_{s}(K)$
is $a=(0,\ldots, 0,a_{0})$
and for every uniformizer $\pi\in K$. 
Here $\sqrt{\overline{\pi^{2}a_{0}}}\in F_{K}^{1/2}$ denotes the square root of the image 
$\overline{\pi^{2}a_{0}}$ of $\pi^{2}a_{0}$ in $F_{K}$.
\item There exists a unique injection $\tilde{\phi}^{\prime (2)}\colon \gr_{2}\hok \rightarrow \gr_{2} \oko \otimes_{F_{K}} F_{K}^{1/2}$ such that the following diagram is commutative for every $s\ge 0$:
\begin{equation}
\label{spcfinj}
\xymatrix{
\gr_{2}W_{s}(K) \ar[rr]^{\tilde{\varphi}^{\prime (2)}_{s}} \ar[rd]_{\delta_{s}^{\prime (2)}} & & \gr_{2}\Omega^{1}_{K}\otimes_{F_{K}}F_{K}^{1/2}\\
& \gr_{2}H^{1}(K,\mathbf{Q}/\mathbf{Z}). \ar[ru]_{\tilde{\phi}^{\prime (2)}} &} 
\end{equation}
\end{enumerate}
\end{prop}

\begin{proof}
By Lemma \ref{gr1l}, we may assume that $s=1$.

(i) Let $a$ be an element of $\fil_{2}K$
and $\pi$ a uniformizer of $K$.
Since $p=2$, we have $\fil_{2}K=\fillog_{2}K$ by (\ref{filpwdef}).
Hence we have $\pi^{2}a\in \dvr_{K}$ by (\ref{filwdef}).
Since $-d(\fil_{2}K)\subset \fil_{2}\Omega_{K}^{1}$,
we have $-da+\sqrt{\overline{\pi^{2}a}}d\pi/\pi^{2}
\in \gr_{2} \oko \otimes_{F_{K}} F_{K}^{1/2}$.
If $a\in \fil_{1}K$, we have $a\in \dvr_{K}$ by (\ref{filwdef}) and (\ref{filzwfilpow}).
Since $-d(\fil_{1}K)\subset \fil_{1}\Omega^{1}_{K}$,
we have
$-da+\sqrt{\overline{\pi^{2}a}}d\pi/\pi^{2}=0$ in $\gr_{2} \oko \otimes_{F_{K}} F_{K}^{1/2}$.
For $a,b\in \fil_{2}K$, we have $\sqrt{\overline{\pi^{2}(a+b)}}=\sqrt{\overline{\pi^{2}a}}+
\sqrt{\overline{\pi^{2}b}}$, since $p=2$.

We prove that $\sqrt{\overline{\pi^{2}a}}d\pi/\pi^{2}$ is
independent of the choice of a uniformizer $\pi$ of $K$. 
Let $u\in\dvr_{K}^{\times}$ be a unit.
Then, in $\gr_{2}\Omega^{1}_{K}\otimes_{F_{K}}F_{K}^{1/2}$,
we have
\begin{equation}
\sqrt{\overline{(u\pi)^{2}a}}d(u\pi )/(u\pi)^{2} = u\sqrt{\overline{\pi^{2}a}}ud\pi/(u\pi)^{2} 
= \sqrt{\overline{\pi^{2}a}}d\pi/\pi^{2}. \notag
\end{equation}
Hence the assertion follows.

(ii) Since $p=2$ and $\fil_{2}K= \fillog_{2}K$, we have $\fil_{2}K \cap (F-1)(K)
=(F-1)(\fillog_{1}K)$ by Lemma \ref{lempbbyf} (iii).
Hence it is sufficient to prove that
$\Ker\tilde{\varphi}_{1}^{\prime (2)}$ is the image of
$(F-1)(\fillog_{1}K)$ in $\gr_{2}K$.

Let $a$ be an element of $\fillog_{1}K$.
By (\ref{filwdef}), we may put $a=a'/\pi$, where $a'\in \dvr_{K}$.
Then we have 
\begin{equation}
\label{eqimasa}
\tilde{\varphi}_{1}^{\prime (2)}(\bar{a}^{2}-\bar{a})=-\bar{a'}d\pi/\pi^{2}+\sqrt{\bar{a'}^{2}}d\pi/\pi^{2}=0.
\end{equation}

Conversely, let $a\in \fil_{2}K$ be a lift of an element of $\Ker\tilde{\varphi}_{1}^{\prime (2)}$.
Since $\fil_{2}K= \fillog_{2}K$, we can put $a=a'/\pi^{2}$, where $a'\in \dvr_{K}$, by (\ref{filwdef}).
Suppose that $\ord_{K}(a')>0$, that is $a\in \fillog_{1}W_{s}(K)$.
Since $\tilde{\varphi}_{1}^{\prime (2)}(\bar{a})
=-(a'\pi^{-1})d\pi/\pi^{2}=0$,
we have $a'\pi^{-1}=0$ in $F_{K}$.
Hence $a\in \fillog_{0}K=\fil_{1}K$, that is $\bar{a}=0$ in $\gr_{2}K$.

Assume that $a'\in \dvr_{K}^{\times}$ is a unit.
Since we have
\begin{equation}
\label{phitto}
\tilde{\varphi}_{1}^{\prime (2)}(\bar{a})=-da+\sqrt{\bar{a'}}d\pi/\pi^{2}
=0, 
\end{equation}
we have $\sqrt{\bar{a'}}\in F_{K}$.
Hence there exist a unit $a''\in \dvr_{K}^{\times}$ and an element $b\in \fillog_{1}K$ such that $a=(F-1)(a''/\pi)+b$.
By (\ref{eqimasa}) and (\ref{phitto}), we have $\tilde{\varphi}_{1}^{\prime (2)}(\bar{b})=0$.
Hence we have $b\in \fil_{1}K$ by the case where $\ord_{K}(a')>0$, which is proved above.
Therefore $\bar{a}\in\gr_{2}K$ is the image of an element of $(F-1)(\fillog_{1}K)$. 
\end{proof}

Let $m\ge 2$ be an integer.
By abuse of notation, we write 
\begin{equation}
\label{phipgen}
\phi'^{(m)}\colon \gr_{m}H^{1}(K,\mathbf{Q}/\mathbf{Z})\rightarrow \gr_{m}\Omega^{1}_{K}
\otimes_{F_{K}}F_{K}^{1/p}
\end{equation}
for the composition of $\phi'^{(m)}$ in (\ref{charforminj}) and
the inclusion $\gr_{m}\Omega^{1}_{K}\rightarrow \gr_{m}\Omega^{1}_{K}\otimes_{F_{K}}F_{K}^{1/p}$ 
if $(p,m)\neq(2,2)$
and $\tilde{\phi}'^{(2)}$ in Proposition \ref{propnrar} (ii) if $(p,m)=(2,2)$.

\begin{defn}
\label{localcharform}
Let $\chi$ be an element of $\hok$.
We put $m=\dt(\chi)$ and assume that $m\ge 2$.
We define the \textit{characteristic form} $\cform(\chi)\in \gr_{m}\Omega^{1}_{K}\otimes_{F_{K}}F_{K}^{1/p}$ of $\chi$
to be the image of $\chi$ by $\phi'^{(m)}$ (\ref{phipgen}).
\end{defn}

By (\ref{charforminj}) and Proposition \ref{propnrar},
we need $F^{1/p}_{K}$ only in the case where $p=2$ and
$\chi \in \fil_{2}H^{1}(K,\mathbf{Q}/\mathbf{Z})- \fillog_{1}H^{1}(K,\mathbf{Q}/\mathbf{Z})$.

In the rest of this subsection, we prepare some lemmas for the proof of Proposition \ref{lemsone}.

\begin{defn}
Let $s\ge 0$ and $r\ge 0$ be integers.
We put $r'=\min\{\ord_{p}(m), s+r\}$ and $s''=\max \{0, r'-r\}$.
We define increasing filtrations $\{\fil^{(r)}_{m}W_{s}(K)\}_{m\in \mathbf{Z}_{\ge 1}}$ 
and $\{\fillog''^{(r)}_{m}W_{s}(K)\}_{m\in \BZ_{\ge 1}}$ of $W_{s}(K)$ by
\begin{align}
\label{filprdef}
\fil^{(r)}_{m}W_{s}(K)&=\fillog_{m-1}^{(r)}W_{s}(K)+V^{s-s''}
\fillog_{m}^{(r)}W_{s''}(K), \\ 
\label{filpprwk}
\fillog_{m}''^{(r)}W_{s}(K)&=\fillog_{[(m-1)/p]}^{(r)}W_{s}(K)
+V^{s-s''}\fillog_{[m/p]}^{(r)}W_{s''}(K).
\end{align}
If $r=0$, then we simply write $\fillog''_{m}W_{s}(K)$ for $\fillog_{m}''^{(0)}W_{s}(K)$.
\end{defn}

If $r=0$, since $s''=s' =\min \{\ord_{p} (m),s\}$, we have $\fillog_{m}'^{(0)}W_{s}(K)=\fil_{m}W_{s}(K)$.
Further we have
\begin{equation}
\label{filppwk}
\fillog''_{m}W_{s}(K)=\fillog_{[(m-1)/p]}W_{s}(K)+V^{s-s'}\fillog_{[m/p]}W_{s'}(K).
\end{equation}

\begin{lem}
\label{lemfilprep}
Let $r\ge 0$ and $0\le t\le s$ be integers.
Let $\pr_{t}\colon W_{s}(K)\rightarrow W_{t}(K)$ be as in (\ref{prt}).
Let $m\ge 1$ be an integer.
\begin{enumerate}
\item $\pr_{t}(\fil_{m}W_{s}(K))=\fil^{(s-t)}_{m}W_{t}(K)$. 
\item We have the exact sequence
\begin{equation}
\label{esprvnl}
0\rightarrow \fil_{m}W_{s-t}(K) \xrightarrow{V^{t}} \fil_{m}W_{s}(K)
\xrightarrow{\pr_{t}} \fillog_{m}'^{(s-t)}W_{t}(K)\rightarrow 0.
\end{equation} 
\item $\pr_{t}(\fillog''_{m}W_{s}(K))=\fillog_{m}''^{(s-t)}W_{t}(K)$.
\item We have the exact sequence
\begin{equation}
\label{esprvpp}
0\rightarrow \fillog''_{m}W_{s-t}(K) \xrightarrow{V^{t}} \fillog''_{m}W_{s}(K)
\xrightarrow{\pr_{t}} \fillog_{m}''^{(s-t)}W_{t}(K)\rightarrow 0. 
\end{equation}
\item $\fillog_{m}''^{(r)}W_{s}(K)=(F-1)^{-1}(\fil^{(r)}_{m}W_{s}(K))$. 
Especially, $\fillog''_{m}W_{s}(K)=(F-1)^{-1}(\fil_{m}W_{s}(K))$.
\end{enumerate}
\end{lem}

\begin{proof}
We put $s'=\min\{\ord_{p}(m),s\}$,
$r'=\min\{\ord_{p}(m), s+r\}$, and $s''=\max \{0, r'-r\}$.

(i) By (\ref{filprdef}), we have
$\fillog_{m}'^{(s-t)}W_{t}(K)=\fillog_{m-1}^{(s-t)}W_{t}(K)$ if $t\le s-s'$ and
$\fillog_{m}'^{(s-t)}W_{t}(K)=
\fillog_{m-1}^{(s-t)}W_{t}(K)+V^{s-s'}\fillog_{m}^{(s-t)}W_{t-s+s'}(K)$ if $t>s-s'$.
By Lemma \ref{lemfilrep} (i), we have $\pr_{t}(\fillog_{m-1}W_{s}(K))=\fillog_{m-1}^{(s-t)}W_{t}(K)$
and,
if $t>s-s'$, we have $\pr_{t}(V^{s-s'}\fillog_{m}W_{s'}(K))=V^{s-s'}\fillog_{m}^{(s-t)}W_{t-s+s'}(K)$.
Hence the assertion follows by (\ref{filpwdef}).

(ii) The assertion follows by (\ref{filpwdef}) and (i).

(iii) The assertion follows similarly as the proof of (i) by (\ref{filpprwk}) and (\ref{filppwk}).

(iv) The assertion follows by (\ref{filppwk}) and (iii).

(v) Since $V^{s-s''}$ and $\pr_{s-s''}$ commute with $F-1$, the morphisms $V^{s-s''}\colon
W_{s''}(K)\rightarrow W_{s}(K)$ and $\pr_{s-s''}\colon W_{s}(K)\rightarrow W_{s-s''}(K)$ induce
$V^{s-s''}\colon (F-1)^{-1}(\fillog_{m}^{(r)}W_{s''}(K))\rightarrow (F-1)^{-1}(\fillog'^{(r)}_{m}W_{s}(K))$
and $\pr_{s-s''}\colon (F-1)^{-1}(\fillog'^{(r)}_{m}W_{s}(K))\rightarrow 
(F-1)^{-1}(\fillog'^{(r+s'')}_{m-1}W_{s-s''}(K))$ respectively.

We prove that $\fillog_{m}''^{(r)}W_{s}(K)\subset (F-1)^{-1}(\fillog_{m}'^{(r)}W_{s}(K))$.
By (\ref{filrdef}) and (\ref{filpprwk}), we have
$\fillog_{m}''^{(r)}W_{s}(K)=\fillog_{[[(m-1)/p]/p^{r}]}W_{s}(K)+V^{s-s''}\fillog_{[[m/p]/p^{r}]}W_{s''}(K)$.
By (\ref{filrdef}) and (\ref{filprdef}), we have
$\fillog_{m}'^{(r)}W_{s}(K)=\fillog_{[(m-1)/p^{r}]}W_{s}(K)+V^{s-s''}\fillog_{[m/p^{r}]}W_{s''}(K)$.
Hence, by Lemma \ref{lemgaus} (ii) and Lemma \ref{lempbbyf} (iii),
we have $\fillog_{m}''^{(r)}W_{s}(K)\subset (F-1)^{-1}(\fillog_{m}'^{(r)}W_{s}(K))$.

We consider the commutative diagram
\begin{equation}
\xymatrix{
\fillog^{(r)}_{[m/p]}W_{s''}(K)\ar[r]^-{V^{s-s''}} \ar[d] & 
\fillog_{m}''^{(r)}W_{s}(K) \ar[r]^-{\pr_{s-s''}} \ar[d]& 
\fillog_{[(m-1)/p]}^{(r+s'')}W_{s-s''}(K)\ar[r] \ar[d]& 0 \\
(F-1)^{-1}(\fillog_{m}^{(r)}W_{s''}(K)) \ar[r]^-{V^{s-s''}} &
(F-1)^{-1}(\fillog'^{(r)}_{m}W_{s}(K)) \ar[r]^-{\pr_{s-s''}} & (F-1)^{-1}(\fillog^{(r+s'')}_{m-1}W_{s-s''}(K)),  &
}
\notag
\end{equation}
where the left and right vertical arrows are the identities by Lemma \ref{lemfilrep} (ii), 
the middle vertical arrow is the inclusion, and the lower horizontal line is exact.
Since the upper horizontal line is exact by Lemma \ref{lemfilrep} (i) and (\ref{filpprwk}),
the assertion follows by applying the snake lemma.
\end{proof}

\begin{cor}
\label{cornlexk}
Let $m\ge 2$ and $0\le t\le s$ be integers.
\begin{enumerate}
\item The exact sequence (\ref{esprvnl}) induces the exact sequence
\begin{equation}
0\rightarrow \gr_{m}W_{s-t}(K)\xrightarrow{\bar{V}^{t}} \gr_{m}W_{s}(K)
\xrightarrow{\overline{\pr}_{t}} \gr^{(s-t)}_{m}W_{t}(K)\rightarrow 0. \notag
\end{equation}
\item The exact sequence (\ref{esprvpp}) induces the exact sequence
\begin{equation}
0\rightarrow \grlog_{m}''W_{s-t}(K)\xrightarrow{\bar{V}^{t}}\grlog_{m}''W_{s}(K)\xrightarrow{\overline{\pr}_{t}}\grlog_{m}''^{(s-t)}W_{t}(K)\rightarrow 0. \notag
\end{equation}
\end{enumerate}
\end{cor}

\begin{proof}
The assertion follows similarly as the proof of Lemma \ref{lemesgrk}. 
\end{proof}

Let $m\ge 2$ be an integer.
By abuse of notation, let 
\begin{equation}
\varphi'^{(m)}_{s}\colon \gr_{m}W_{s}(K)\rightarrow \gr_{m}\Omega^{1}_{K}\otimes_{F_{K}}F_{K}^{1/p}
\notag
\end{equation}
be the composition of $\varphi'^{(m)}_{s}$ (\ref{vphipk}) and the inclusion 
$\gr_{m}\Omega^{1}_{K}\rightarrow \gr_{m}\Omega_{K}^{1}\otimes_{F_{K}}F_{K}^{1/p}$
if $(p,m)\neq (2,2)$ and 
$\tilde{\varphi}'^{(2)}_{s}$ in Proposition \ref{propnrar} (i) if $(p,m)=(2,2)$.

Let $r\ge 0$ be an integer.
By Lemma \ref{lemfilprep} (v), the morphism $F-1\colon W_{s}(K)\rightarrow W_{s}(K)$
induces the injection
\begin{equation}
\overline{F-1}\colon \grlog''^{(r)}_{m}W_{s}(K)\rightarrow \grlog'^{(r)}_{m}W_{s}(K).  \notag
\end{equation}
Especially, the morphism $F-1$ induces the injection
\begin{equation}
\overline{F-1}\colon \grlog''_{m}W_{s}(K)\rightarrow \grlog'_{m}W_{s}(K).  \notag
\end{equation}

\begin{lem}[{cf.\ \cite[Proposition 3.2.1, Proposition 3.2.3]{ma}}]
\label{lemexas}
Let $m\ge 2$ be an integer.
Then we have the exact sequence
\begin{equation}
0\rightarrow \grlog''_{m}W_{s}(K)\xrightarrow{\overline{F-1}}
\gr_{m}W_{s}(K) \xrightarrow{\varphi_{s}'^{(m)}} \gr_{m}\Omega_{K}^{1}\otimes_{F}F^{1/p}.
\notag
\end{equation}
\end{lem}

\begin{proof}
As in the proof of \cite[Proposition 3.2.1]{ma} and Proposition \ref{propnrar} (ii),
the morphism $\varphi_{s}'^{(m)}$ factors through
\begin{equation}
\gr_{m}H^{1}(K,\BZ/p^{s}\BZ)\simeq \fil_{m}W_{s}(K)/((F-1)(W_{s}(K))\cap \fil_{m}W_{s}(K)+
\fil_{m-1}W_{s}(K)).\notag
\end{equation}
Since this factorization defines the injection $\phi'^{(m)}$ by \cite[Proposition 3.2.3]{ma}
and Proposition \ref{propnrar} (ii), the assertion follows. 
\end{proof}

\begin{lem}
\label{lemfilp}
Let $m\ge 1$ and $r\ge 0$ be integers.
\begin{enumerate}
\item $\fillog_{m}'^{(r)}K=\fillog_{m/p^{r}}K$ if $m\in p^{r+1}\BZ$
and $\fillog_{m}'^{(r)}K=\fillog_{[(m-1)/p^{r}]}K$ if $m\notin p^{r+1}\BZ$.
\item $\fillog_{m}''^{(r)}K=\fillog_{[m/p^{r+1}]}K$.
\end{enumerate}
\end{lem}

\begin{proof}
(i) By (\ref{filprdef}), we have $\fillog_{m}'^{(r)}K=\fillog_{m}^{(r)}K$ if $m\in p^{r+1}\BZ$
and $\fillog_{m}'^{(r)}K=\fillog_{m-1}^{(r)}K$ if $m\notin p^{r+1}\BZ$.
Hence the assertion follows by (\ref{filrdef}). 

(ii) By Lemma \ref{lemfilprep} (v), we have $\fillog_{m}''^{(r)}K=(F-1)^{-1}(\fillog_{m}'^{(r)}K)$.
By (i) and Lemma \ref{lempbbyf} (iii), we have 
$\fillog_{m}''^{(r)}K=\fillog_{m/p^{r+1}}W_{s}(K)$ if $m\in p^{r+1}\BZ$ and 
$\fillog_{m}''^{(r)}K=\fillog_{[[(m-1)/p^{r}]/p]}W_{s}(K)$ if $m\notin p^{r+1}\BZ$.
Hence the assertion follows by Lemma \ref{lemgaus}.
\end{proof}

\begin{cor}
\label{corgrpk}
Let $m\ge 2$ and $r\ge 0$ be integers.
\begin{enumerate}
\item Assume that $r\ge 1$.
Then $\gr^{(r)}_{m}K=\grlog_{[m/p^{r}]}K$ if $m\in p^{r+1}\BZ$ or 
$\ord_{p}(m-1)=r$,
and $\gr^{(r)}_{m}K=0$ if otherwise.
\item $\grlog_{m}''^{(r)}K=\grlog_{m/p^{r+1}}K$ if $m\in p^{r+1}\BZ$, and
$\grlog_{m}''^{(r)}K=0$ if $m\notin p^{r+1}\BZ$.
\end{enumerate}
\end{cor}

\begin{proof}
(i) Assume that $m\in p^{r+1}\BZ$.
Since $r\ge 1$, we have $m-1\notin p^{r}\BZ$.
Hence $\grlog_{m}'^{(r)}K=\fillog_{[m/p^{r}]}K/\fillog_{[(m-2)/p^{r}]}K$
by Lemma \ref{lemfilp} (i).
By Lemma \ref{lemgaus} (i), the assertion follows in this case.

Assume that $m\notin p^{r+1}\BZ$.
By Lemma \ref{lemfilp} (i), we have
$\grlog_{m}'^{(r)}K=
\fillog_{[(m-1)/p^{r}]}K/\fillog_{[(m-2)/p^{r}]}K$ if $m-1\notin p^{r+1}\BZ$ and
$\grlog_{m}'^{(r)}K=0$ if $m-1\in p^{r+1}\BZ$.
Suppose that $m-1\notin p^{r+1}\BZ$.
By Lemma \ref{lemgaus} (i), we have $\grlog_{m}'^{(r)}K=\grlog_{[(m-1)/p^{r}]}K$ if $m-1\in p^{r}\BZ$ 
and $\grlog_{m}'^{(r)}K=0$ if $m-1\notin p^{r}\BZ$.
If $m-1\in p^{r}\BZ$, then we have $m\notin p^{r}\BZ$, since $r\ge 1$.
Hence the assertion follows by Lemma \ref{lemgaus} (i).

(ii) By Lemma \ref{lemfilp} (ii), we have 
$\grlog_{m}''^{(r)}K=\fillog_{[m/p^{r+1}]}''^{(r)}K/\fillog_{[(m-1)/p^{r+1}]}''^{(r)}K$.
Hence the assertion follows by Lemma \ref{lemgaus} (i).
\end{proof}

We note that if $r=0$ and if $m\in p\BZ$ 
then $\grlog_{m}'^{(r)}K=\gr_{m}K=\fillog_{m}K/\fillog_{m-2}K$.

\subsection{Sheafification: logarithmic case}
\label{sssheaf}

Let $X$ be a smooth separated scheme over a perfect field $k$ of 
characteristic $p>0$.
Let $D$ be a divisor on $X$ with simple normal crossings and $\{D_{i}\}_{i\in I}$ the irreducible components of $D$.
The generic point of $D_{i}$ is denoted by $\mathfrak{p}_{i}$ for $i\in I$.
We put $U=X-D$ and 
let $j\colon U\rightarrow X$ be the canonical open immersion.
For $i\in I$, let $\dvr_{K_{i}}$ denote the completion $\cdvr_{X,\mathfrak{p}_{i}}$ of the local ring $\dvr_{X,\mathfrak{p}_{i}}$ at
$\mathfrak{p}_{i}$ and $K_{i}$ the fractional field of $\dvr_{K_{i}}$
called {\it local field} at $\mathfrak{p}_{i}$.

Let $\epsilon\colon X_{\et}\rightarrow X_{\mathrm{Zar}}$
be the canonical mapping from the \'{e}tale site of $X$ to
the Zariski site of $X$. 
We use the same notation $j_{*}$ for the push-forward of both \'{e}tale sheaves and Zariski sheaves.
We consider the exact sequence
\begin{equation}
0\rightarrow W_{s}(\mathbf{F}_{p})\rightarrow W_{s}(\dvr_{U_{\et}})\xrightarrow{F-1} W_{s}(\dvr_{U_{\et}})\rightarrow 0 \notag
\end{equation}
of \'{e}tale sheaves on $U$ for $s\in \BZ_{\ge 0}$.
Since $R^{1}(\epsilon\circ j)_{*}W_{s}(\dvr_{U_{\et}})=0$, we have an exact sequence
\begin{equation}
\label{eqszshf}
0\rightarrow j_{*}W_{s}(\mathbf{F}_{p})\rightarrow j_{*}W_{s}(\dvr_{U})
\xrightarrow{F-1} j_{*}W_{s}(\dvr_{U}) \rightarrow 
R^{1}(\epsilon\circ j)_{*}\mathbf{Z}/p^{s}\mathbf{Z}
\rightarrow 0 
\end{equation}
We write 
\begin{equation}
\label{deltsshf}
\delta_{s}\colon j_{*}W_{s}(\dvr_{U})\rightarrow 
R^{1}(\epsilon\circ j)_{*}\BZ/p^{s}\BZ 
\end{equation}
for the forth morphism in (\ref{eqszshf}).

Let $V$ be an open subset of $X$.
Since we have the spectral sequence $E_{2}^{p,q}=H^{p}_{\mathrm{Zar}}(V,
R^{q}(\epsilon \circ j)_{*}\BZ/p^{s}\BZ)\Rightarrow H^{p+q}_{\et}(U\cap V,\BZ/p^{s}\BZ)$ and $E_{2}^{1,0}=E_{2}^{2,0}=0$, 
the canonical morphism
\begin{equation}
H^{1}_{\et}(U\cap V,\mathbf{Z}/p^{s}\BZ)\rightarrow \Gamma(V,R^{1}(\epsilon\circ j)_{*}\BZ/p^{s}\BZ) \notag
\end{equation}
is an isomorphism.
By the exact sequence (\ref{eqszshf}), the morphism $\delta_{s}$ (\ref{deltsshf}) 
induces an isomorphism
\begin{equation}
j_{*}W_{s}(\dvr_{U})/(F-1)j_{*}W_{s}(\dvr_{U})\rightarrow
R^{1}(\epsilon\circ j)_{*}\BZ/p^{s}\BZ. \notag
\end{equation}

If $D_{i}\cap V\neq \emptyset$ and if $a\in \Gamma(U\cap V,W_{s}(\dvr_{U}))$,
let $a|_{K_{i}}$ denote the image of $a$ by
\begin{equation}
\Gamma(U\cap V,W_{s}(\dvr_{U}))\rightarrow W_{s}(K_{i}).\notag
\end{equation}
Similarly, if $D_{i}\cap V\neq \emptyset$ and if $\chi\in H^{1}_{\et}(U\cap V, \mathbf{Z}/p^{s}\BZ)$,
let $\chi|_{K_{i}}$ denote the image of $\chi$ by
\begin{equation}
H^{1}_{\et}(U\cap V,\BZ/p^{s}\BZ)\rightarrow H^{1}(K_{i},\BZ/p^{s}\BZ).
\notag
\end{equation}

\begin{defn}
Let $R=\sum_{i\in I}n_{i}D_{i}$, where $n_{i}\in \mathbf{Z}_{\ge 0}$ for $i\in I$,
and let $j_{i}\colon \Spec K_{i}\rightarrow X$ denote the canonical morphism for $i\in I$. 
\begin{enumerate}
\item We define a subsheaf $\fillog_{R} j_{\ast}W_{s}(\dvr_{U})$
of Zariski sheaf $j_{\ast}W_{s}(\dvr_{U})$ to be
the pull-back of $\bigoplus_{i\in I}j_{i*}\fillog_{n_{i}}W_{s}(K_{i})$ by the morphism $j_{*}W_{s}(\dvr_{U})\rightarrow \bigoplus_{i\in I}
j_{i*}W_{s}(K_{i})$.
\item We define a subsheaf $\fillog_{R}R^{1}(\epsilon\circ j)_{\ast}\BZ/p^{s}\BZ$
of $R^{1}(\epsilon\circ j)_{\ast}\BZ/p^{s}\BZ$ to be the image of $\fillog_{R}j_{\ast}W_{s}(\dvr_{U})$ by $\delta_{s}$ (\ref{deltsshf}).
\item We define a subsheaf $\fillog_{R}j_{\ast}\Omega^{1}_{U}$
of $j_{*}\Omega_{U}^{1}$ to be $\Omega^{1}_{X}(\log D)(R)$.
\end{enumerate}
\end{defn}

We consider the morphism
\begin{equation}
\label{fsds}
-F^{s-1}d\colon j_{\ast}W_{s}(\dvr_{U})\rightarrow j_{\ast}\Omega^{1}_{U} 
\; ; \; (a_{s-1},\ldots, a_{0})\mapsto -\sum_{i=0}^{s-1}a_{i}^{p^{i}-1}da_{i}. 
\end{equation}
Let $R=\sum_{i\in I}n_{i}D_{i}$,
where $n_{i}\in \mathbf{Z}_{\ge 0}$ for $i\in I$.
Then (\ref{fsds}) induces the morphism 
\begin{equation}
\fillog_{R}j_{\ast}W_{s}(\dvr_{U})
\rightarrow \fillog_{R}j_{\ast}\Omega^{1}_{U}. \notag
\end{equation} 
Let $R^{\prime}=\sum_{i\in I}n_{i}^{\prime}D_{i}$, where $n_{i}^{\prime}\in \mathbf{Z}_{\ge 0}$ such that $n_{i}'\le n_{i}$ for $i\in I$.
Then we have $\fillog_{R}\supset \fillog_{R^{\prime}}$ and
put $\grlog_{R/R^{\prime}}=\fillog_{R}/\fillog_{R^{\prime}}$.
Then the morphism (\ref{fsds}) induces the morphism
\begin{equation}
\label{phishf}
\varphi^{(R/R^{\prime})}_{s}\colon
\grlog_{R/R^{\prime}}j_{\ast}W_{s}(\dvr_{U})\rightarrow \grlog_{R/R^{\prime}}j_{\ast}\Omega^{1}_{U}.
\end{equation}
If $R=R'+D_{i}$ for some $i\in I$,
then we simply write $\varphi_{s}^{(R,i)}$ for $\varphi_{s}^{(R,R^{\prime})}$ and $\grlog_{R,i}$ for $\grlog_{R/R^{\prime}}$.

Let $0\le t\le s$ be integers.
We put $[R/p^{j}]=\sum_{i\in I}[n_{i}/p^{j}]D_{i}$.
We consider the projection 
\begin{equation}
\label{prts}
\pr_{t}\colon j_{\ast}W_{s}(\dvr_{U})\rightarrow j_{\ast}W_{t}(\dvr_{U})
\; ; \; (a_{s-1},\ldots,a_{0})\mapsto (a_{s-1},\ldots,a_{s-t}). 
\end{equation}
Since we have $\pr_{t}(\fillog_{R}j_{*}W_{s}(\dvr_{U}))
=\fillog_{[R/p^{s-t}]}j_{*}W_{t}(\dvr_{U})$ by (\ref{filrdef}) and Lemma \ref{lemfilrep} (i), 
we have the exact sequence
\begin{equation}
\label{esprvlshf}
0\rightarrow \fillog_{R}j_{*}W_{s-t}(\dvr_{U})\xrightarrow{V^{t}} 
\fillog_{R}j_{*}W_{s}(\dvr_{U}) \xrightarrow{\pr_{t}} 
\fillog_{[R/p^{s-t}]}j_{*}W_{t}(\dvr_{U}) \rightarrow 0.
\end{equation}

\begin{lem}
\label{lemesgrwk}
Let $R=\sum_{i\in I}n_{i}D_{i}$ and
$R^{\prime}=\sum_{i\in I}n_{i}'D_{i}$, where $n_{i},n_{i}'\in \mathbf{Z}_{\ge 0}$
and $n_{i}'\le n_{i}$ for every $i\in I$.
Then the exact sequence (\ref{esprvlshf}) induces the exact sequence 
\begin{equation}
\label{esgrlogshf}
0\rightarrow \grlog_{R/R^{\prime}}j_{\ast}W_{s-t}(\dvr_{U})\xrightarrow{\bar{V}^{t}} \grlog_{R/R^{\prime}}j_{\ast}W_{s}(\dvr_{U})
\xrightarrow{\overline{\pr}_{t}} \grlog_{[R/p^{s-t}]/[R'/p^{s-t}]}j_{\ast}W_{t}(\dvr_{U})\rightarrow 0.
\end{equation}
Especially, if $R=R'+D_{i}$ for some $i\in I$, we have the exact sequence
\begin{equation}
0\rightarrow \grlog_{R,i}j_{\ast}W_{s-t}(\dvr_{U})\xrightarrow{\bar{V}^{t}} \grlog_{R,i}j_{\ast}W_{s}(\dvr_{U})
\xrightarrow{\overline{\pr}_{t}} \grlog_{[R/p^{s-t}]/[(R-D_{i})/p^{s-t}]}j_{\ast}W_{t}(\dvr_{U})\rightarrow 0. \notag
\end{equation}
\end{lem}

\begin{proof}
The assertion follows similarly as the proof of Lemma \ref{lemesgrk}.
In fact, we consider the commutative diagram
\begin{equation}
\label{cdgflogshf}
\xymatrix{
0 \ar[r] & \fillog_{R'}j_{*}W_{s-t}(\dvr_{U}) \ar[r]^-{V^{t}} \ar[d] & \fillog_{R'}j_{*}W_{s}(\dvr_{U}) \ar[r] ^-{\pr_{t}}
\ar[d] & \fillog_{[R'/p^{s-t}]}j_{*}W_{t}(\dvr_{U}) \ar[r] \ar[d] & 0 \\
0 \ar[r] & \fillog_{R}j_{*}W_{s-t}(\dvr_{U}) \ar[r]^-{V^{t}}  & \fillog_{R}j_{*}W_{s}(\dvr_{U}) \ar[r] ^-{\pr_{t}}
& \fillog_{[R/p^{s-t}]}j_{*}W_{t}(\dvr_{U}) \ar[r] & 0,
} 
\end{equation}
where the horizontal lines are exact and the vertical arrows are inclusions.
Then this diagram induces the sequence (\ref{esgrlogshf}). 
By taking stalks of (\ref{cdgflogshf}), the exactness of (\ref{esgrlogshf}) follows.
\end{proof}

Let $R=\sum_{i\in I}n_{i}D_{i}$ and $R'=\sum_{i\in I}n_{i}'D_{i}$,
where $n_{i},n_{i}'\in \mathbf{Z}_{\ge 0}$ and $n_{i}'\le n_{i}$ for every $i\in I$.
We consider the morphism
\begin{equation}
\label{barfshf}
\bar{F}\colon \grlog_{[R/p]/[R^{\prime}/p]}j_{\ast}W_{s}(\dvr_{U})\rightarrow \grlog_{R/R^{\prime}}j_{\ast}W_{s}(\dvr_{U})
\end{equation}
induced by the Frobenius $F\colon j_{*}W_{s}(\dvr_{U})\rightarrow j_{*}W_{s}(\dvr_{U})$.
Since $F^{-1}(\fillog_{R}j_{*}W_{s}(\dvr_{U}))=\fillog_{[R/p]}j_{*}W_{s}(\dvr_{U})$
by Lemma \ref{lempbbyf} (iii) and similarly for $R'$, 
the morphism (\ref{barfshf}) is injective.

We consider the morphism
\begin{equation}
\label{barfmoshf}
\overline{F-1}\colon \grlog_{[R/p]/[R^{\prime}/p]}j_{\ast}W_{s}(\dvr_{U})\rightarrow \grlog_{R/R^{\prime}}j_{\ast}W_{s}(\dvr_{U})
\end{equation}
induced by $F-1\colon j_{*}W_{s}(\dvr_{U})\rightarrow j_{*}W_{s}(\dvr_{U})$.
If $R=R'+D_{i}$ for some $i\in I$, then
the morphisms (\ref{barfshf}) and (\ref{barfmoshf}) are the same, since $[R/p]\le R'$ by product order.

\begin{lem}
\label{lemaina}
Let $A$ be a smooth ring over $k$.
Let $t_{1},\ldots,t_{r}$ be elements of $A$
such that $(t_{1}\cdots t_{r}=0)$ is a divisor on $\Spec A$ with simple normal crossings whose irreducible components are $\{(t_{i}=0)\}_{i=1}^{r}$.
Let $a$ be an element of $\Frac A$. 
Assume that $a^{p}t_{1}^{n_{1}}\cdots t_{r}^{n_{r}}\in A$,
where $n_{1},\ldots,n_{r}$ are integers such that $0\le n_{i}<p$ for $i=1,\ldots ,r$.
Then we have $a\in A$.
\end{lem}

\begin{proof}
Since $a^{p}t_{1}^{n_{1}}\cdots t_{r}^{n_{r}}\in A$,
the valuation of $a^{p}t_{1}^{n_{1}}\cdots t_{r}^{n_{r}}$ in
$A_{(t_{i})}$ is non-negative for $i=1,\ldots,r$.
Since the normalized valuation of $a^{p}$ in $\Frac A_{(t_{i})}$ for $i=1,\ldots,r$ is divided by $p$ and $0\le n_{i}<p$ for $i=1,\ldots ,r$,
the valuation of $a$ in $\Frac A_{(t_{i})}$ for $i=1,\ldots,r$
is non-negative.
Since $A$ is factorial, we have $A[1/t_{1}\cdots t_{r}]\cap \bigcap_{i=1}^{r}A_{(t_{i})}=A$.
Hence the assertion follows.
\end{proof}

\begin{lem}
\label{lemcap}
Let $\mf$, $\mg$, and $\mh$ be sheaves on $X$
and let $\mf_{i}$, $\mg_{i}$, and $\mh_{i}$ be subsheaves of $\mf$, $\mg$, and $\mh$ respectively
for $i=1,2,3$.
Assume that $\mf_{3}=\mf_{1}\cap \mf_{2}$, $\mh_{3}=\mh_{1}\cap \mh_{2}$,
and that $\mg_{3}\subset \mg_{1}\cap \mg_{2}$.
If we have an exact sequence $0\rightarrow \mf \rightarrow \mg \rightarrow \mh \rightarrow 0$
and if this exact sequence induces the exact sequence $0\rightarrow \mf_{i}\rightarrow \mg_{i}
\rightarrow \mh_{i}\rightarrow 0$ for $i=1,2,3$, 
then we have $\mg_{3}=\mg_{1}\cap \mg_{2}$.
\end{lem}

\begin{proof}
We consider the commutative diagram
\begin{equation}
\label{cdfgh}
\xymatrix{
& 0 \ar[d] & 0 \ar[d]  & 0 \ar[d] & \\
0 \ar[r] & \mf_{3} \ar[r] \ar[d] & \mg_{3} \ar[r] \ar[d] & \mh_{3} \ar[r] \ar[d] & 0 \\
0 \ar[r] & \mf_{1}\oplus \mf_{2} \ar[r] \ar[d] & \mg_{1}\oplus \mg_{2} \ar[r] \ar[d] & 
\mh_{1}\oplus \mh_{2} \ar[r] \ar[d] & 0 \\
0 \ar[r] & \mf \ar[r] & \mg \ar[r] & \mh \ar[r] & 0,
} 
\end{equation}
where the bottom vertical arrows are defined by the difference.
Since $\mf_{3}=\mf_{1}\cap \mf_{2}$ and $\mh_{3}=\mh_{1}\cap \mh_{2}$,
the left and right vertical columns are exact.
By applying the snake lemma to the lower two lines, 
we have $\mg_{3}=\mg_{1}\cap \mg_{2}$.
\end{proof}

\begin{prop}
\label{lemsone}
Let $R=\sum_{i\in I}n_{i}D_{i}$, where $n_{i}\in \mathbf{Z}_{\ge 0}$ for $i\in I$.
Let $s\ge 0$ be an integer and let $i$ be an element of $I$ such that $n_{i}\ge 1$.
We put $R'=R-D_{i}$.
Then we have the exact sequence
\begin{equation}
0\rightarrow \fillog_{[R/p]}j_{\ast}W_{s}(\dvr_{U})/\fillog_{[R'/p]}j_{*}W_{s}(\dvr_{U})
\xrightarrow{\bar{F}} \grlog_{R,i}j_{\ast}W_{s}(\dvr_{U})
\xrightarrow{\varphi_{s}^{(R,i)}} \grlog_{R,i}j_{\ast}\Omega^{1}_{U},
\notag
\end{equation}
where $\fillog_{[R/p]}j_{\ast}W_{s}(\dvr_{U})/\fillog_{[R'/p]}j_{*}W_{s}(\dvr_{U})$ is 
$\grlog_{[R/p],i}j_{*}W_{s}(\dvr_{U})$ if $n_{i}\in p\BZ$ and $0$ if $n_{i} \notin p\BZ$.
\end{prop}

\begin{proof}
We may assume that $s\ge 1$, $I=\{1,\ldots,r\}$, and that $i=1$.
Let $j_{1}\colon \Spec K_{1}\rightarrow X$ be the canonical morphism.
We consider the commutative diagram
\begin{equation}
\label{cdgrwksh}
\xymatrix{
0 \ar[r] & \fillog_{[R/p],1}j_{*}W_{s}(\dvr_{U})/\fillog_{[R'/p],1}j_{*}W_{s}(\dvr_{U}) \ar[r]^-{\bar{F}} \ar[d] &
\grlog_{R,1}j_{*}W_{s}(\dvr_{U}) 
\ar[r]^-{\varphi_{s}^{(R,1)}} \ar[d] & \grlog_{R,1}j_{*}\Omega^{1}_{U} \ar[d] \\
0 \ar[r] & j_{1*}(\fillog_{[n_{1}/p]}W_{s}(K_{1})/\fillog_{[(n_{1}-1)/p]}W_{s}(K_{1})) \ar[r]^-{\bar{F}} &
j_{1*}\grlog_{n_{1}}W_{s}(K_{1}) \ar[r]^-{\varphi_{s}^{(n_{1})}} &
j_{1*}\grlog_{n_{1}}\Omega^{1}_{K_{1}},
}
\end{equation}
where the vertical arrows are inclusions.
Since the lower line is exact by Lemma \ref{exseqgrwk}, 
it is sufficient to prove that the left square in (\ref{cdgrwksh}) is cartesian.

If $n_{1}\notin p\BZ$, then the assertion follows since 
$\fillog_{[R/p],1}j_{*}W_{s}(\dvr_{U})/\fillog_{[R'/p],1}j_{*}W_{s}(\dvr_{U})=0$ and 
$\fillog_{[n_{1}/p]}W_{s}(K_{1})/\fillog_{[(n_{1}-1)/p]}W_{s}(K_{1})=0$ by Lemma \ref{lemgaus} (i).

Assume that $n_{1}\in p\BZ$.
Then we have $\fillog_{[R/p],1}j_{*}W_{s}(\dvr_{U})/\fillog_{[R'/p],1}j_{*}W_{s}(\dvr_{U})=
\grlog_{[R/p],1}j_{*}W_{s}(\dvr_{U})$ and
$\fillog_{[n_{1}/p]}W_{s}(K_{1})/\fillog_{[(n_{1}-1)/p]}W_{s}(K_{1})=\grlog_{n_{1}/p}W_{s}(K_{1})$
by Lemma \ref{lemgaus} (i).

We prove the assertion by the induction on $s$.
Suppose that $s=1$.
Since the assertion is local,
we may assume that $X=\Spec A$ is affine and that $D_{i}=(t_{i}=0)$ for $i\in I$, 
where $t_{i}\in A$ for $i\in I$.
Further we may assume that the invertible $\dvr_{D_{1}}$-modules
$\grlog_{R,1}j_{*}\dvr_{U}$ 
and $\grlog_{[R/p],1}j_{*}\dvr_{U}$
are gererated by $c_{0}=1/t_{1}^{n_{1}}\cdots t_{r}^{n_{r}}$ and
$c_{1}=1/t_{1}^{n_{1}/p}t_{2}^{m_{2}'}\cdots t_{r}^{m_{r}'}$ respectively,
where $m_{i}'=[n_{i}/p]$ for $i\in I-\{1\}$.
Let $k(D_{1})$ denote the functional field of $D_{1}$.
We identify $\grlog_{n_{1}}K_{1}$ with
$k(D_{1})\cdot c_{0}$
and $\grlog_{n_{1}/p}K_{1}$ with 
$k(D_{1})\cdot c_{1}$.

Let $\bar{a}$ be an element of $k(D_{1})$
such that $\bar{F}(\bar{a}c_{1})=\bar{a}^{p}c_{1}^{p}\in \grlog_{R,1}j_{*}\dvr_{U}$.
Since $(\bar{a}^{p}c_{1}^{p}/c_{0})\cdot c_{0}\in 
\grlog_{R,1}j_{*}\dvr_{U}= 
\dvr_{D_{1}}\cdot c_{0}$,
we have $\bar{a}^{p}c_{1}^{p}/c_{0}\in \dvr_{D_{1}}$.
Since $c_{1}^{p}/c_{0}=t_{2}^{n_{2}-pm_{2}'}\cdots t_{r}^{n_{r}-pm_{r}'}$
and $0\le n_{i}-pm_{i}^{\prime} <p$
for $i\in I-\{1\}$,
we have $\bar{a}\in \dvr_{D_{1}}$ by Lemma \ref{lemaina}.
Hence we have $\bar{a}c_{1}\in \dvr_{D_{1}}\cdot c_{1}=\grlog_{[R/p],1}j_{*}\dvr_{U}$.
Hence the assertion follows if $s=1$.

If $s> 1$, we put $\mf=j_{1*}\grlog_{n_{1}}W_{s-1}(K_{1})$, $\mf_{1}=\grlog_{R,1}j_{*}W_{s-1}(\dvr_{U})$, 
$\mf_{2}=j_{1*}\grlog_{n_{1}/p}W_{s-1}(K_{1})$, and 
$\mf_{3}=\grlog_{[R/p],1}j_{*}W_{s-1}(\dvr_{U})$.
Since the canonical morphisms $\mf_{1}\rightarrow \mf$ and $\mf_{3}\rightarrow \mf_{2}$ are injective
and both $\bar{F}\colon \mf_{3}\rightarrow \mf_{1}$ and $\bar{F}\colon \mf_{2}\rightarrow \mf$ 
are injective, we may identify $\mf_{i}$ with a subsheaf of $\mf$ for $i=1,2,3$.
We also put $\mg=j_{1*}\grlog_{n_{1}}W_{s}(K_{1})$, $\mg_{1}=\grlog_{R,1}j_{*}W_{s}(\dvr_{U})$, 
$\mg_{2}=j_{1*}\grlog_{n_{1}/p}W_{s}(K_{1})$, and 
$\mg_{3}=\grlog_{[R/p],1}j_{*}W_{s}(\dvr_{U})$.
We further put $\mh=j_{1*}(\grlog_{n}^{(s-1)}K_{1})$, 
$\mh_{1}=\grlog_{[R/p^{s-1}]/[R'/p^{s-1}]}j_{*}\dvr_{U}$,
$\mh_{2}=j_{1*}(\grlog^{(s-1)}_{n_{1}/p}K_{1})$, and 
$\mh_{3}=\grlog_{[R/p^{s}]/[R'/p^{s}]}j_{*}\dvr_{U}$.
Similarly as $\mf_{i}$, we may identify $\mg_{i}$ and $\mh_{i}$
with subsheaves of $\mg$ and $\mh$ respectively for $i=1,2,3$.

By the induction hypothesis, we have $\mf_{3}=\mf_{1}\cap \mf_{2}$.
If $n_{1}\notin p^{s}\BZ$, then $\mh_{2}=\mh_{3}=0$
by Lemma \ref{lemgaus} (i) and (\ref{filrdef}).
If $n_{1}\in p^{s}\BZ$, then we have $\mh_{3}=\mh_{1}\cap \mh_{2}$ 
by Lemma \ref{lemgaus} (i), (\ref{filrdef}), and the induction hypothesis.
By the commutativity of (\ref{cdgrwksh}), we have $\mg_{3}\subset \mg_{1}\cap \mg_{2}$.
Since exact sequences in Lemma \ref{lemesgrk} and Lemma \ref{lemesgrwk} 
in the case where $t=1$
are compatible with the inclusions of sheaves above,
the assertion follows by Lemma \ref{lemcap}.
\end{proof}

\begin{lem}
\label{lemshffil}
Let $f\colon \mf\rightarrow \mg$ be a surjection of sheaves on $X$.
Let $g\colon \mg \rightarrow \mathcal{H}$ be a morphism of sheaves on $X$.
We put $\Gamma=(\BZ_{\ge 0})^{r}$, where $r>0$ is an integer, and
let $1_{i}\in \Gamma$ be the element whose $i$-th component is $1$ and the others are $0$
for $i=1,\ldots,r$.
Let $\{\fillog_{n}\mf\}_{n\in \Gamma}$ and $\{\fillog_{n}\mathcal{H}\}_{n\in \Gamma}$
be increasing filtrations of $\mf$ and $\mathcal{H}$ respectively by product order.
Assume that $\bigcup_{n\in \Gamma}\fillog_{n}\mf=\mf$ and $\bigcup_{n\in \Gamma}
\fillog_{n}\mathcal{H}=\mathcal{H}$.
We put $\fillog_{n}\mg=f(\fillog_{n}\mf)$ for $n\in \Gamma$,
which define an increasing filtration of $\mg$.
If $g(\fillog_{n}\mg)\subset \fillog_{n}\mh$ for every $n\in \Gamma$ and if
the morphism
$\fillog_{n+1_{i}}\mg/\fillog_{n}\mg\rightarrow \fillog_{n+1_{i}}\mathcal{H}/\fillog_{n}\mathcal{H}$ 
induced by $g$ is injective for every $n\in \Gamma$ and $i=1,\ldots,r$, 
then we have $\fillog_{n}\mg=g^{-1}(\fillog_{n}\mathcal{H})$ for every $n\in \Gamma$.
\end{lem}

\begin{proof}
Let $n\in \Gamma$ be an element.
We prove that the morphism $\mg/\fillog_{n}\mg\rightarrow \mathcal{H}/\fillog_{n}\mathcal{H}$ is injective.
Since $\mf=\bigcup_{n\in \Gamma}\fillog_{n}\mf$ and $f$ is surjective, we have
$\mg=\bigcup_{n\in \Gamma}\fillog_{n}\mg$ and 
hence $\mg/\fillog_{n}\mg=\varinjlim_{n'}\fillog_{n'}\mg/\fillog_{n}\mg$, where $n'$ rums through
the elements of $\Gamma$ greater than $n$ by product order.
Since $\mh=\bigcup_{n\in \Gamma}\fillog_{n}\mh$, we have
$\mh/\fillog_{n}\mh=\varinjlim_{n'}\fillog_{n'}\mh/\fillog_{n}\mh$, where $n'$ rums through
the elements of $\Gamma$ greater than $n$.
Hence it is sufficient to prove that $\fillog_{n'}\mg/\fillog_{n}\mg\rightarrow \fillog_{n'}\mathcal{H}/\fillog_{n}\mathcal{H}$ is injective for every $n'\in \Gamma$ such that $n'\ge n$.
We prove this assertion by the induction on $n'$.

If $n'=n$, the assertion follows since $\fillog_{n'}\mg/\fillog_{n}\mg=0$ 
and $\fillog_{n'}\mh/\fillog_{n}\mh=0$.
For $n'>n$, take $i$ such that $n'-1_{i}\ge n$.
We consider the commutative diagram
\begin{equation}
\xymatrix{
0\ar[r] & \fillog_{n'-1_{i}}\mg/\fillog_{n}\mg \ar[r] \ar[d] &
\fillog_{n'}\mg/\fillog_{n}\mg \ar[r] \ar[d] & \fillog_{n'}\mg/\fillog_{n'-1_{i}}\mg \ar[r] \ar[d] & 0 \\
0\ar[r] & \fillog_{n'-1_{i}}\mathcal{H}/\fillog_{n}\mathcal{H} \ar[r]  &
\fillog_{n'}\mathcal{H}/\fillog_{n}\mathcal{H} \ar[r] & 
\fillog_{n'}\mathcal{H}/\fillog_{n'-1_{i}}\mathcal{H} \ar[r] & 0,
} \notag
\end{equation}
where the horizontal lines are exact.
By the induction hypothesis, the left vertical arrow is injective.
Since the right vertical arrow is injective, the middle vertical arrow is injective.
Hence the assertion follows.
\end{proof}

\begin{prop}
\label{proplogcd}
Let $R=\sum_{i\in I}n_{i}D_{i}$, where $n_{i}\in \BZ_{\ge 0}$ for $i\in I$. 
Let $j_{i}\colon \Spec K_{i}\rightarrow X$ be the canonical morphism for $i\in I$.
\begin{enumerate}
\item The subsheaf $\fillog_{R}R^{1}(\epsilon\circ j)_{*}\BZ/p^{s}\BZ$ is equal to the pull-back of
$\bigoplus_{i\in I}j_{i*}\fillog_{n_{i}}H^{1}(K_{i}, \BQ/\BZ)$ by the morphism 
$R^{1}(\epsilon \circ j)_{*}\BZ/p^{s}\BZ\rightarrow \bigoplus_{i\in I}j_{i*}H^{1}(K_{i},\BQ/\BZ)$.
\item Let $R^{\prime}=\sum_{i\in I}n_{i}'D_{i}$, where $n_{i}'\in \BZ_{\ge 0}$ and $n_{i}-1\le n_{i}'\le n_{i}$ for $i\in I$.
Then there exists a unique injection $\phi_{s}^{(R/R^{\prime})}\colon \grlog_{R/R'}R^{1}(\epsilon\circ j)_{\ast}\BZ/p^{s}\BZ\rightarrow \grlog_{R/R'}
j_{*}\Omega^{1}_{U}$ such that the following diagram is commutative:
\begin{equation}
\label{cdcflsh}
\xymatrix{
\grlog_{R/R'}j_{\ast}W_{s}(\dvr_{U}) \ar[dr]_-{\delta_{s}^{(R/R')}}
\ar[rr]^-{\varphi_{s}^{(R/R')}} & & \grlog_{R/R'}j_{\ast}\Omega^{1}_{U}\\
 & \grlog_{R/R'}R^{1}(\epsilon\circ j)_{\ast}\BZ/p^{s}\BZ. \ar[ur]_-{\phi_{s}^{(R/R')}} & 
}
\end{equation} 
\end{enumerate}
\end{prop}

\begin{proof}
Let $i$ be an element of $I$ such that $n_{i}\ge 2$.
Since the kernel of $\delta_{s}^{(R,i)}$ is the image of $\overline{F-1}$ (\ref{barfmoshf}) and
the morphisms $\bar{F}$ (\ref{barfshf}) and $\overline{F-1}$ (\ref{barfmoshf}) are the same,
the kernel of $\delta_{s}^{(R,i)}$ is equal to the kernel of $\varphi_{s}^{(R,i)}$
by Proposition \ref{lemsone}.
Since $\delta_{s}^{(R,i)}$ is surjective,
there exists a unique injection $\phi_{s}^{(R,i)}\colon 
\grlog_{R,i}R^{1}(\epsilon\circ j)_{\ast}\BZ/p^{s}\BZ\rightarrow
\grlog_{R,i}j_{\ast}\Omega^{1}_{U}$ such that the diagram (\ref{cdcflsh}) 
for $R'=R-D_{i}$ is commutative.

(i) Let $i$ be an element of $I$ such that $n_{i}\ge 2$. 
We consider the commutative diagram
\begin{equation}
\xymatrix{
\grlog_{R,i}R^{1}(\epsilon\circ j)_{*}\BZ/p^{s}\BZ \ar[r] \ar[d]_-{\phi_{s}^{(R,i)}}  & 
j_{i*}\grlog_{n_{i}}H^{1}(K_{i}, \BQ/\BZ) \ar[d]^-{\phi^{(n_{i})}} \\
\grlog_{R,i}j_{*}\Omega_{U}^{1} \ar[r] & j_{i*}\grlog_{n_{i}}\Omega^{1}_{K_{i}},
} \notag
\end{equation}
where the lower horizontal arrow is the inclusion and $\phi^{(n_{i})}$ is as in (\ref{refinedswaninj}).
Since the left vertical arrow is injective as proved above, the upper horizontal arrow is injective. 
Hence the assertion follows by applying Lemma \ref{lemshffil}
to the case where $\mf=j_{*}W_{s}(\dvr_{U})$, $\mg=R^{1}(\epsilon\circ j)_{*}\BZ/p^{s}\BZ$, 
and $\mathcal{H}=\bigoplus_{i\in I}j_{i*}H^{1}(K_{i},\BQ/\BZ)$.

(ii) Let $J$ be the subset of $I$ consisting of $i\in I$ such that $n_{i}'\neq n_{i}$.
We consider the commutative diagram
\begin{equation}
\xymatrix{
\grlog_{R/R'}j_{*}W_{s}(\dvr_{U}) \ar[rr]^-{\varphi_{s}^{(R/R')}} \ar[d]_-{\delta_{s}^{(R/R')}}
& & \grlog_{R/R'}j_{*}\Omega_{U}^{1} \ar[d] \\
\grlog_{R/R'}R^{1}(\epsilon\circ j)_{*}\BZ/p^{s}\BZ \ar[r] &
\bigoplus_{i\in J}j_{i*}\grlog_{n_{i}}H^{1}(K_{i}, \BQ/\BZ) \ar[r]^-{\oplus\phi^{(n_{i})}} &
\bigoplus_{i\in J}j_{i*}\grlog_{n_{i}}\Omega_{K_{i}}^{1},
} \notag
\end{equation}
where $\phi^{(n_{i})}$ is as in (\ref{refinedswaninj}) for $i\in J$.
By (i), the left lower horizontal arrow is injective.
Since $\grlog_{n_{i}}\Omega^{1}_{K_{i}}$ is the stalk of $\grlog_{R/R'}j_{*}\Omega_{U}^{1}$
at the generic point of $D_{i}$ for $i\in J$, the kernel of the canonical morphism
$\fillog_{R}j_{*}\Omega_{U}^{1}\rightarrow \bigoplus_{i\in J}j_{i*}\grlog_{n_{i}}\Omega_{K_{i}}^{1}$
is the intersection of $\fillog_{R-D_{i}}j_{*}\Omega_{U}^{1}$ for $i\in J$.
Hence the right vertical arrow is injective.
Since the right lower horizontal arrow is injective, the kernel of $\varphi_{s}^{(R/R')}$
is equal to that of $\delta_{s}^{(R/R')}$.
Since $\delta_{s}^{(R/R')}$ is surjective, the assertion follows.
\end{proof}

\begin{defn}
Let $\chi$ be an element of $H^{1}_{\et}(U,\BQ/\BZ)$.
We define the {\it Swan conductor divisor} $R_{\chi}$ of $\chi$
by $R_{\chi}=\sum_{i\in I}\sw(\chi|_{K_{i}})D_{i}$.
\end{defn}

\begin{defn}
\label{defrsw}
Let $\chi$ be an element of $H^{1}_{\et}(U,\BQ/\BZ)$.
Assume that $\sw(\chi|_{K_{i}})>0$ for some $i\in I$.
Let $p^{s}$ be the order of the $p$-part of $\chi$.
We put $Z=\Supp(R_{\chi})$. 
We define the {\it refined Swan conductor} $\rsw(\chi)$ of $\chi$
to be the image of the $p$-part of $\chi$ by the composition
\begin{align}
\Gamma (X,\fillog_{R_{\chi}}R^{1}(\epsilon\circ&j)_{*}\mathbf{Z}/p^{s}\mathbf{Z})\rightarrow \Gamma(X,\grlog_{R_{\chi}/(R_{\chi}-Z)}R^{1}(\epsilon\circ j)_{*}\BZ/p^{s}\BZ) \notag\\ 
&\xrightarrow{\phi_{s}^{(R_{\chi}/(R_{\chi}-Z))}(X)} 
\Gamma(X,\grlog_{R_{\chi}/(R_{\chi}-Z)}j_{*}\Omega^{1}_{U})=\Gamma(Z,\Omega^{1}_{X}(\log D)(R_{\chi})\otimes_{\dvr_{X}}\dvr_{Z}). \notag
\end{align}
\end{defn}

By the construction of $\phi_{s}^{(R_{\chi}/(R_{\chi}-Z))}$,
the germ $\rsw(\chi)_{\mathfrak{p}_{i}}$ of $\rsw(\chi)$
at the generic point $\mathfrak{p}_{i}$ of $D_{i}$ contained in $Z$ is equal to $\rsw(\chi|_{K_{i}})$.
This implies that $\rsw(\chi)$ in Definition \ref{defrsw} is none other than the refined Swan conductor of $\chi$
in the sense of \cite[(3.4.2)]{ka2}.

\subsection{Sheafification: non-logarithmic case}

We recall the definition of the radicial covering $S^{1/p}$ of a scheme
$S$ over a perfect field $k$ of characteristic $p>0$.
We consider the commutative diagram 
\begin{equation}
\xymatrix{
S^{1/p} \ar[d] \ar[r] & S \ar[r]^{F_{S}} \ar[d] & S \ar[d]\\
\Spec k \ar[r]_{F^{-1}_{k}} & \Spec k \ar[r]_{F_{k}}& \Spec k,}
\notag
\end{equation}
where the left square is the base change over $k$ by the inverse $F_{k}^{-1}$ of $F_{k}$.
The symbols $F_{S}$ and $F_{k}$ denote the absolute Frobenius of $S$ and $\Spec k$ respectively.
We define the \textit{radicial covering} $S^{1/p} \rightarrow S$ by the composition of morphisms in the upper line.

We keep the notation in Subsection \ref{sssheaf}. 

\begin{defn}
Let $R=\sum_{i\in I}n_{i}D_{i}$, where $n_{i}\in \mathbf{Z}_{\ge 1}$ for $i\in I$,
and let $j_{i}\colon \Spec K_{i}\rightarrow X$ denote the canonical morphism for $i\in I$. 
Let $r\ge 0$ be an integer.
\begin{enumerate}
\item We define subsheaves $\fil^{(r)}_{R}j_{*}W_{s}(\dvr_{U})$ 
and $\fillog''^{(r)}_{R}j_{*}W_{s}(\dvr_{U})$ of
Zariski sheaf $j_{*}W_{s}(\dvr_{U})$
to be
the pull-back of $\bigoplus_{i\in I}j_{i*}\fil^{(r)}_{n_{i}}W_{s}(K_{i})$ 
and $\bigoplus_{i\in I}j_{i*}\fillog''^{(r)}_{n_{i}}W_{s}(K_{i})$
by the morphism $j_{*}W_{s}(\dvr_{U})\rightarrow \bigoplus_{i\in I}
j_{i*}W_{s}(K_{i})$ respectively.

If $r=0$, then we simply write $\fil_{R}j_{*}W_{s}(\dvr_{U})$ and $\fillog_{R}''j_{*}W_{s}(\dvr_{U})$
for $\fil^{(0)}_{R}j_{*}W_{s}(\dvr_{U})$ and $\fillog_{R}''^{(0)}j_{*}W_{s}(\dvr_{U})$ respectively. 
\item We define a subsheaf $\fil_{R}R^{1}(\epsilon\circ j)_{\ast}\BZ/p^{s}\BZ$
of $R^{1}(\epsilon\circ j)_{\ast}\BZ/p^{s}\BZ$ to be the image of $\fil_{R}j_{\ast}W_{s}(\dvr_{U})$ by $\delta_{s}$ (\ref{deltsshf}).
\item We define a subsheaf $\fil_{R}j_{\ast}\Omega^{1}_{U}$
of $j_{*}\Omega_{U}^{1}$ to be $\Omega^{1}_{X}(R)$.
\end{enumerate}
\end{defn}

Similarly as in the logarithmic case, we consider the morphism $-F^{s-1}d$ (\ref{fsds}).
Let $R=\sum_{i\in I}n_{i}D_{i}$, where $n_{i}\in \BZ_{\ge 1}$ for $i\in I$.
Then $-F^{s-1}d$ (\ref{fsds}) induces the morphism
\begin{equation}
\fil_{R}j_{\ast}W_{s}(\dvr_{U})\rightarrow \fil_{R}j_{*}\Omega^{1}_{U}. \notag
\end{equation}
For $R'=\sum_{i\in I}n_{i}'D_{i}$, where $n_{i}'\in \BZ_{\ge 1}$ such that
$n_{i}'\le n_{i}$ for $i\in I$,
we put $\gr_{R/R'}=\fil_{R}/\fil_{R'}$.
Then the morphism $-F^{s-1}d$ (\ref{fsds}) induces the morphism
\begin{equation}
\label{phipshf}
\varphi'^{(R/R')}_{s}\colon \gr_{R/R'}j_{*}W_{s}(\dvr_{U})\rightarrow \gr_{R/R'}j_{\ast}\Omega^{1}_{U}.
\end{equation}
Let $D^{(R/R')}=R-R'\subset D$.
If $p\neq 2$ or there is no $i\in I$ such that $(n_{i},n_{i}')= (2,1)$,
let $\tilde{\varphi}'^{(R/R')}_{s}\colon \grlog_{R/R'}'j_{*}W_{s}(\dvr_{U})
\rightarrow \grlog_{R/R'}'j_{*}\Omega_{U}^{1}\otimes_{\dvr_{D^{(R/R')}}}
\dvr_{D^{(R/R')^{1/2}}}$ be the composition
\begin{equation}
\gr_{R/R'}j_{*}W_{s}(\dvr_{U})\xrightarrow{\varphi'^{(R/R')}_{s}} \gr_{R/R'}j_{\ast}\Omega^{1}_{U}\rightarrow
\grlog_{R/R'}'j_{*}\Omega_{U}^{1}\otimes_{\dvr_{D^{(R/R')}}}
\dvr_{D^{(R/R')^{1/2}}}.\notag
\end{equation} 
If otherwise, as in the proof of Proposition \ref{propnrar} (i),
there exists a unique morphism
\begin{equation}
\label{tildphi}
\tilde{\varphi}'^{(R/R')}_{s}\colon \grlog_{R/R'}'j_{*}W_{s}(\dvr_{U})
\rightarrow \grlog_{R/R'}'j_{*}\Omega_{U}^{1}\otimes_{\dvr_{D^{(R/R')}}}
\dvr_{D^{(R/R')^{1/p}}} 
\end{equation}
such that locally $\tilde{\varphi}'^{(R/R')}_{s}(\bar{a})=-\sum_{i=0}^{s-1}a_{i}^{p^{i}-1}da_{i}
+\sum_{(n_{i},n_{i}')=(2,1)}\sqrt{\overline{t_{i}^{2}a_{0}}}dt_{i}/t_{i}^{2}$
for every $\bar{a}\in \gr_{R/R'}j_{*}W_{s}(\dvr_{U})$ whose lift is $a=(a_{s-1},\ldots,a_{0})\in 
\fil_{R}j_{*}W_{s}(\dvr_{U})$ and for every local equation $t_{i}$ of $D_{i}$ for $i\in I$ such that
$(n_{i},n_{i}')=(2,1)$.

If $R=R'+D_{i}$ for some $i\in I$, then
we simply write $\gr_{R,i}$ for $\gr_{R/R'}$, $\tilde{\varphi}'^{(R,i)}$ for $\tilde{\varphi}'^{(R/R')}$,
and similarly for $\grlog''_{R/R'}$, $\grlog'^{(r)}_{R/R'}$, and $\grlog_{R/R'}''^{(r)}$.

\begin{lem}
\label{lempbasw}
Let $R=\sum_{i\in I}n_{i}D_{i}$, where $n_{i}\in \BZ_{\ge 1}$ for $i\in I$, and let $r\ge 0$ be an integer.
Then we have $\fillog_{R}''^{(r)}j_{*}W_{s}(\dvr_{U})=(F-1)^{-1}(\fillog_{R}'^{(r)}j_{*}W_{s}(\dvr_{U}))$.
Especially, we have $\fillog_{R}''j_{*}W_{s}(\dvr_{U})=(F-1)^{-1}(\fil_{R}j_{*}W_{s}(\dvr_{U}))$.
\end{lem}

\begin{proof}
Let $j_{i}\colon \Spec K_{i}\rightarrow X$ be the canonical morphism for $i\in I$.
Since $F-1$ is compatible with the canonical morphism $j_{*}W_{s}(\dvr_{U})
\rightarrow \bigoplus_{i\in I}j_{i*}W_{s}(K_{i})$, the assertions follow 
by Lemma \ref{lemfilprep} (v).
\end{proof}

\begin{lem}
\label{lemfilrsh}
Let $r\ge 0$ be an integer.
Let $R=\sum_{i\in I}n_{i}D_{i}$ and $R'=\sum_{i\in I}n_{i}'D_{i}$, where $n_{i},n_{i}'\in \BZ_{\ge 1}$
such that $n_{i}'=n_{i}/p^{r}$ if $n_{i}\in p^{r+1}\BZ$ and $n_{i}'=[(n_{i}-1)/p^{r}]$ if $n_{i}\notin
p^{r+1}\BZ$ for every $i\in I$.
\begin{enumerate}
\item $\fillog_{R}'^{(r)}j_{*}\dvr_{U}=\fillog_{R'}j_{*}\dvr_{U}$.
\item $\fillog_{R}''^{(r)}j_{*}\dvr_{U}=\fillog_{[R/p^{r+1}]}j_{*}\dvr_{U}$.
\end{enumerate}
\end{lem}

\begin{proof}
The assertions follow by Lemma \ref{lemfilp}.
\end{proof}

\begin{cor}
\label{corgrpshf}
Let the notation be as in Lemma \ref{lemfilrsh}.
Let $i$ be an element of $I$ such that $n_{i}\ge 2$.
\begin{enumerate}
\item Assume that $r\ge 1$.
Then $\gr^{(r)}_{R,i}j_{*}\dvr_{U}=\grlog_{R',i}j_{*}\dvr_{U}$ 
if $n_{i}\in p^{r+1}\BZ$ or $\ord_{p}(n_{i}-1)=r$,
and $\gr^{(r)}_{R,i}j_{*}\dvr_{U}=0$ if otherwise.
\item $\grlog''^{(r)}_{R,i}j_{*}\dvr_{U}=\grlog_{[R/p^{r+1}],i}j_{*}\dvr_{U}=\grlog_{[R'/p]}j_{*}\dvr_{U}$
if $n_{i}\in p^{r+1}\BZ$, and
$\grlog''^{(r)}_{R,i}j_{*}\dvr_{U}=0$ if $n_{i}\notin p^{r+1}\BZ$.
\end{enumerate}
\end{cor}

\begin{proof}
Since $[R/p^{r+1}]=[R'/p]$ by Lemma \ref{lemgaus},
the assertions follow by Corollary \ref{corgrpk} and Lemma \ref{lemfilrsh}.
\end{proof}

Let $R=\sum_{i\in I}n_{i}D_{i}$
and $R'=\sum_{i\in I}n_{i}'D_{i}$, 
where $n_{i}, n_{i}'\in \mathbf{Z}_{\ge 1}$ and $n_{i}'\le n_{i}$ for every $i\in I$.
Let $0\le t\le s$ be integers.
Since we have $\pr_{t}(\fil_{R}j_{*}W_{s}(\dvr_{U}))=\fillog_{R}'^{(s-t)}j_{*}W_{t}(\dvr_{U})$ by
Lemma \ref{lemfilprep} (i),
we have the exact sequence
\begin{equation}
\label{esprvnlsh}
0\rightarrow \fil_{R}j_{*}W_{s-t}(\dvr_{U})\xrightarrow{V^{t}} \fil_{R}j_{*}W_{s}(\dvr_{U})
\xrightarrow{\pr_{t}} \fillog_{R}'^{(s-t)}j_{*}W_{t}(\dvr_{U})\rightarrow 0. 
\end{equation}
Similarly, since $\pr_{t}(\fillog_{R}''j_{*}W_{s}(\dvr_{U}))=\fillog_{R}''^{(s-t)}j_{*}W_{t}(\dvr_{U})$
by Lemma \ref{lemfilprep} (iii),
we have the exact sequence
\begin{equation}
\label{esprvppsh}
0\rightarrow \fillog_{R}''j_{*}W_{s-t}(\dvr_{U})\xrightarrow{V^{t}} \fillog''_{R}j_{*}W_{s}(\dvr_{U})
\xrightarrow{\pr_{t}} \fillog_{R}''^{(s-t)}j_{*}W_{t}(\dvr_{U})\rightarrow 0.
\end{equation}

\begin{lem}
\label{lemvtprt}
Let $R=\sum_{i\in I}n_{i}D_{i}$ and
$R'=\sum_{i\in I}n_{i}'D_{i}$, where $n_{i},n_{i}'\in \BZ_{\ge 1}$
and $n_{i}-1\le n_{i}'\le n_{i}$ for every $i\in I$.
Let $0\le t\le s$ be integers.
\begin{enumerate}
\item The exact sequence (\ref{esprvnlsh}) induces the exact sequence
\begin{equation}
0\rightarrow \gr_{R/R'}j_{*}W_{s-t}(\dvr_{U})\xrightarrow{\bar{V}^{t}}
\gr_{R/R'}j_{*}W_{s}(\dvr_{U})\xrightarrow{\overline{\pr}_{t}}
\grlog'^{(s-t)}_{R/R'}j_{*}W_{t}(\dvr_{U})\rightarrow 0. \notag
\end{equation}
\item The exact sequence (\ref{esprvppsh}) induces the exact sequence
\begin{equation}
0\rightarrow \grlog''_{R/R'}j_{*}W_{s-t}(\dvr_{U})\xrightarrow{\bar{V}^{t}}
\grlog''_{R/R'}j_{*}W_{s}(\dvr_{U})\xrightarrow{\overline{\pr}_{t}}
\grlog''^{(s-t)}_{R/R'}j_{*}W_{t}(\dvr_{U})\rightarrow 0. \notag
\end{equation}
\end{enumerate}
\end{lem}

\begin{proof}
The assertions follow similarly as the proof of Lemma \ref{lemesgrwk}.
\end{proof}

Let $r\ge 0$ be an integer.
By Lemma \ref{lempbasw}, 
the morphism $F-1\colon j_{*}W_{s}(\dvr_{U})\rightarrow j_{*}W_{s}(\dvr_{U})$ 
induces the injection
\begin{equation}
\overline{F-1}\colon \grlog''^{(r)}_{R/R'}j_{*}W_{s}(\dvr_{U})\rightarrow
\grlog'^{(r)}_{R/R'}j_{*}W_{s}(\dvr_{U}). \notag
\end{equation}
Especially, the morphism $F-1$ induces the injection
\begin{equation}
\overline{F-1}\colon \grlog''_{R/R'}j_{*}W_{t}(\dvr_{U})\rightarrow
\grlog'_{R/R'}j_{*}W_{t}(\dvr_{U}). \notag
\end{equation}

\begin{lem}
\label{lemnlogex}
Let $R=\sum_{i\in I}n_{i}D_{i}$, where $n_{i}\in \BZ_{\ge 1}$ for $i\in I$.
Let $s\ge 0$ be an integer and let $i$ be an element of $I$ such that $n_{i}\ge 2$.
Then we have the exact sequence
\begin{equation}
0\rightarrow \grlog''_{R,i}j_{*}W_{s}(\dvr_{U})\xrightarrow{\overline{F-1}} \gr_{R,i}j_{*}W_{s}(\dvr_{U})
\xrightarrow{\tilde{\varphi}'^{(R,i)}_{s}} \gr_{R,i}j_{*}\Omega^{1}_{U}\otimes_{\dvr_{D_{i}}}
\dvr_{D_{i}^{1/p}}.
\notag
\end{equation}
\end{lem}

\begin{proof}
We may assume that $s\ge 1$, $I=\{1,\ldots,r\}$, and that $i=1$.
Let $j_{1}\colon \Spec K_{1}\rightarrow X$ be the canonical morphism.
We consider the commutative diagram
\begin{equation}
\label{cdexnlog}
\xymatrix{0\ar[r] &
\grlog''_{R,1}j_{*}W_{s}(\dvr_{U})\ar[r]^-{\overline{F-1}} \ar[d] &
\gr_{R,1}j_{*}W_{s}(\dvr_{U}) \ar[r]^-{\varphi'^{(R,1)}_{s}} \ar[d] &
\gr_{R,1}j_{*}\Omega_{U}^{1}\otimes_{\dvr_{D_{1}}}\dvr_{D_{1}^{1/p}} \ar[d]\\
0\ar[r] &
j_{1*}\grlog''_{n_{1}}W_{s}(K_{1})\ar[r]^-{\overline{F-1}} &
j_{1*}\gr_{n_{1}}W_{s}(K_{1}) \ar[r]^-{\varphi'^{(n_{1})}_{s}} &
j_{1*}(\gr_{n_{1}}\Omega^{1}_{K_{1}}\otimes_{F_{K_{1}}}F_{K_{1}}^{1/p}),
}
\end{equation}
where $F_{K_{1}}$ denotes the residue field of $K_{1}$ and
the vertical arrows are canonical injections.
By Lemma \ref{lemexas}, the lower horizontal line is exact.
Hence it is sufficient to prove that the left square
in (\ref{cdexnlog}) is cartesian.

We prove the assertion by the induction on $s$. 
Suppose that $s=1$.
If $n_{1}\notin p\BZ$, then we have $\grlog''_{n_{1}}W_{s}(K_{1})=0$ and
$\grlog''_{R,1}j_{*}\dvr_{U}=0$ by Corollary \ref{corgrpk} (ii) and Corollary \ref{corgrpshf} (ii).
Hence the assertion follows in this case.

Assume that $n_{1}\in p\BZ$.
By (\ref{filpwdef}), we have $\gr_{n_{1}}K_{1}=\fillog_{n_{1}}K_{1}/\fillog_{n_{1}-2}K_{1}$.
By Corollary \ref{corgrpk} (ii),
we have $\grlog''_{n_{1}}K_{1}=\grlog_{n_{1}/p}K_{1}$.
Since the assertion is a local property,
we may assume that $X=\Spec A$ is affine and
that $D_{i}=(t_{i}=0)$ for $i\in I$, where $t_{i}\in A$ for $i\in I$.
Further we may assume that the invertible $\dvr_{2D_{1}}$-module 
$\gr_{R,1}j_{*}\dvr_{U}$ is generated by 
$c_{0}=1/t_{1}^{n_{1}}\cdots t_{r}^{n_{r}}$,
and that the invertible $\dvr_{D_{1}}$-module $\grlog''_{R,1}j_{*}\dvr_{U}$
is generated by $c_{1}=1/t_{1}^{n_{1}/p}t_{2}^{m_{2}'}\cdots t_{r}^{m_{r}'}$, where $m_{i}'=[n_{i}/p]$ 
for $i\in I-\{1\}$.
Let $R(2D_{1})$ denote the stalk of $\dvr_{2D_{1}}$ at the generic point of $2D_{1}$
and let $k(D_{1})$ denote the functional field of $D_{1}$.
Then we may identify $\gr_{n_{1}}K_{1}$ with $R(2D_{1})\cdot c_{0}$
and $\grlog''_{n_{1}}K_{1}$ with $k(D_{1})\cdot c_{1}$.

Let $\bar{a}$ be an element of $k(D_{1})$ such that
$\overline{(F-1)}(\bar{a}c_{1})\in \gr_{R,1}j_{*}\dvr_{U}$.
Since we have $\overline{(F-1)}(\bar{a}c_{1})=((\bar{a}^{p}c_{1}^{p}-\bar{a}c_{1})/c_{0})\cdot c_{0}
\in \gr_{R,1}j_{*}\dvr_{U}=\dvr_{2D_{1}}\cdot c_{0}$, 
we have $(\bar{a}^{p}c_{1}^{p}-\bar{a}c_{1})/c_{0}\in \dvr_{2D_{1}}$.
Since $c_{1}/c_{0}=t_{1}^{n_{1}-n_{1}/p}t_{2}^{n_{2}-m_{2}'}
\cdots t_{r}^{n_{r}-m_{r}'}$ and $n_{1}-n_{1}/p\ge 1$,
we have $(\bar{a}^{p}c_{1}^{p}-\bar{a}c_{1})/c_{0}=\bar{a}^{p}c_{1}^{p}/c_{0}$
in $\dvr_{D_{1}}$.
Since $c_{1}^{p}/c_{0}=t_{2}^{n_{2}-pm_{2}'}\cdots t_{r}^{n_{r}-pm_{r}'}$
and $0\le n_{i}-pm_{i}'< p$ for $i\in I-\{1\}$,
we have $\bar{a}\in \dvr_{D_{1}}$ by Lemma \ref{lemaina}.
Hence we have $\bar{a}c_{1}\in \dvr_{D_{1}}\cdot c_{1}=\grlog''_{R,1}j_{*}\dvr_{U}$.
Thus the assertion follows if $s=1$.

If $s> 1$, we put $\mf=j_{1*}\gr_{n_{1}}W_{s-1}(K_{1})$, $\mf_{1}=\gr_{R,1}j_{*}W_{s-1}(\dvr_{U})$, 
$\mf_{2}=j_{1*}\grlog''_{n_{1}}W_{s-1}(K_{1})$, and 
$\mf_{3}=\grlog''_{R,1}j_{*}W_{s-1}(\dvr_{U})$.
Since the canonical morphisms $\mf_{1}\rightarrow \mf$ and $\mf_{3}\rightarrow \mf_{2}$ are injective
and both $\overline{F-1}\colon \mf_{3}\rightarrow \mf_{1}$ and $\overline{F-1}\colon \mf_{2}\rightarrow
\mf$ are injective, 
we may identify $\mf_{i}$ with a subsheaf of $\mf$ for $i=1,2,3$.
We also put $\mg=j_{1*}\gr_{n_{1}}W_{s}(K_{1})$, $\mg_{1}=\gr_{R,1}j_{*}W_{s}(\dvr_{U})$, 
$\mg_{2}=j_{1*}\grlog''_{n_{1}}W_{s}(K_{1})$, and 
$\mg_{3}=\grlog''_{R,1}j_{*}W_{s}(\dvr_{U})$.
Further we put $\mh=j_{1*}\grlog_{n_{1}}'^{(s-1)}K_{1}$, $\mh_{1}=\grlog'^{(s-1)}_{R,1}j_{*}\dvr_{U}$,
$\mh_{2}=j_{1*}\grlog''^{(s-1)}_{n_{1}}K_{1}$, and 
$\mh_{3}=\grlog''^{(s-1)}_{R,1}j_{*}\dvr_{U}$.
Similarly as $\mf_{i}$, we may identify $\mg_{i}$ and $\mh_{i}$ with subsheaves
of $\mg$ and $\mh$ respectively for $i=1,2,3$.

By the induction hypothesis, we have $\mf_{3}=\mf_{1}\cap \mf_{2}$.
If $n_{1}\notin p^{s}\BZ$, then we have $\mh_{2}=\mh_{3}=0$ by 
Corollary \ref{corgrpk} (ii) and Corollary \ref{corgrpshf} (ii).
If $n_{1}\in p^{s}\BZ$, then we have $\mh_{3}=\mh_{1}\cap \mh_{2}$
by Corollary \ref{corgrpk}, Corollary \ref{corgrpshf}, and 
the case where $s=1$ in the proof of Proposition \ref{lemsone}.
By the commutativity of (\ref{cdexnlog}), we have $\mg_{3}\subset \mg_{1}\cap \mg_{2}$.
Since exact sequences in Corollary \ref{cornlexk} and Lemma \ref{lemvtprt} 
in the case where $t=1$ are compatible with the inclusions of sheaves above, 
the assertion follows by Lemma \ref{lemcap}.
\end{proof}

\begin{prop}
\label{propnlogcd}
Let $R=\sum_{i\in I}n_{i}D_{i}$, where $n_{i}\in \BZ_{\ge 1}$ for $i\in I$.
Let $j_{i}\colon \Spec K_{i}\rightarrow X$ be the canonical morphism for $i\in I$.
\begin{enumerate}
\item The subsheaf $\fil_{R}R^{1}(\epsilon\circ j)_{*}\BZ/p^{s}\BZ$ is equal to the pull-back of 
$\bigoplus_{i\in I}j_{i*}\fil_{n_{i}}H^{1}(K_{i},\mathbf{Q}/\mathbf{Z})$ by the morphism
$R^{1}(\epsilon\circ j)_{*}\mathbf{Z}/p^{s}\mathbf{Z}
\rightarrow \bigoplus_{i\in I}j_{i*}H^{1}(K_{i},\mathbf{Q}/\mathbf{Z})$.
\item Let $R'=\sum_{i\in I}n'_{i}D_{i}$, where $n'_{i}\in \BZ_{\ge 1}$
such that $n_{i}-1\le n'_{i} \le n_{i}$ for $i\in I$.
Then there exists a unique injection
$\phi_{s}'^{(R/R')}\colon \gr_{R/R'}R^{1}(\epsilon\circ j)_{*}\BZ/p^{s}\BZ
\rightarrow \gr_{R/R'}j_{*}\Omega^{1}_{U}\otimes_{\dvr_{D^{(R/R')}}}\dvr_{D^{(R/R')1/p}}$
such that the following diagram is commutative: 
\begin{equation}
\label{cdcfnlsh}
\xymatrix{\gr_{R/R'}j_{*}W_{s}(\dvr_{U}) \ar[dr]_-{\delta'^{(R/R')}_{s}}
\ar[rr]^-{\tilde{\varphi}'^{(R/R')}_{s}} & & \gr_{R/R'}j_{*}\Omega^{1}_{U}
\otimes_{\dvr_{D^{(R/R')}}}\dvr_{D^{(R/R')1/p}}\\
& \gr_{R/R'}R^{1}(\epsilon\circ j)_{*}\BZ/p^{s}\BZ. \ar[ur]_-{\phi'^{(R/R')}_{s}} &
}
\end{equation}
\end{enumerate}
\end{prop}

\begin{proof}
Let $i$ be an element of $I$ such that $n_{i}\ge 2$.
By Lemma \ref{lemnlogex}, the kernel of $\delta_{s}'^{(R,i)}$ is equal to the kernel of 
$\tilde{\varphi}_{s}'^{(R,i)}$.
Since $\delta_{s}'^{(R,i)}$ is surjective,
there exists a unique injection $\phi'^{(R,i)}_{s}\colon \gr_{R,i}R^{1}(\epsilon\circ j)_{*}\BZ/p^{s}\BZ\rightarrow \gr_{R,i}j_{*}\Omega^{1}_{U}\otimes_{\dvr_{D_{i}}}\dvr_{D_{i}^{1/p}}$ such that the 
diagram (\ref{cdcfnlsh}) for $R'=R-D_{i}$ is commutative.

(i) Let $i$ be an element of $I$ such that $n_{i}\ge 2$. 
We consider the commutative diagram
\begin{equation}
\xymatrix{
\gr_{R,i}R^{1}(\epsilon\circ j)_{*}\BZ/p^{s}\BZ \ar[r] \ar[d]_-{\phi_{s}'^{(R,i)}} &
j_{i*}\grlog_{n_{i}}H^{1}(K_{i},\BQ/\BZ) \ar[d]^-{\phi'^{(n_{i})}} \\
\gr_{R,i}j_{*}\Omega_{U}^{1}\otimes_{\dvr_{D_{i}}}\dvr_{D_{i}^{1/p}} \ar[r] &
j_{i*}(\gr_{n_{i}}\Omega_{K_{i}}^{1}\otimes_{F_{K_{i}}}F_{K_{i}}^{1/p}),  
}\notag
\end{equation}
where $F_{K_{i}}$ is the residue field of $K_{i}$,
the lower horizontal arrow is the inclusion,
and $\phi'^{(n_{i})}$ is as in (\ref{phipgen}).
Since the left vertical arrow is injective as proved above, 
the upper horizontal arrow is injective.
Hence the assertion follows by Lemma \ref{lemshffil}
similarly as the proof of Proposition \ref{proplogcd} (i).

(ii) Let $J$ be the subset of $I$ consisting of $i\in I$ such that $n_{i}'\neq n_{i}$.
Since $\gr_{n_{i}}j_{i*}\Omega_{K_{i}}^{1}\otimes_{F_{K_{i}}}F_{K_{i}}^{1/p}$ is
the stalk of $\gr_{R/R'}j_{*}\Omega^{1}_{U}\otimes_{\dvr_{D^{(R/R')}}}\dvr_{D^{(R/R')1/p}}$ 
at the generic point of $D_{i}^{1/p}$ for $i\in J$,
the assertion follows similarly as the proof of Proposition \ref{proplogcd} (ii).
\end{proof}

\begin{defn}
\label{defdtdiv}
Let $\chi$ be an element of $H^{1}_{\et}(U,\BQ/\BZ)$.
We define the \textit{total dimension divisor} $R'_{\chi}$ of $\chi$
by $R'_{\chi}=\sum_{i\in I}\dt(\chi|_{K_{i}})D_{i}$.
\end{defn}

We note that we have $\Supp (R'_{\chi}-D)=\Supp (R_{\chi})$ by (\ref{filzwfilpow}).

\begin{defn}
Let $\chi$ be an element of $H^{1}_{\et}(U,\BQ/\BZ)$.
Assume that $\dt(\chi|_{K_{i}})>1$ for some $i\in I$.
Let $p^{s}$ be the order of the $p$-part of $\chi$.
We put $Z=\Supp (R_{\chi}'-D)$.
We define the \textit{characteristic form} $\cform(\chi)$ of $\chi$
to be the image of the $p$-part of $\chi$ by the composition
\begin{align}
\Gamma(X,\fil_{R_{\chi}}R^{1}&(\epsilon\circ j)_{*}\BZ/p^{s}\BZ)\rightarrow \Gamma(X,\gr_{R'_{\chi}/
(R'_{\chi}-Z)}R^{1}(\epsilon\circ j)_{*}\BZ/p^{s}\BZ) \notag \\
&\xrightarrow{\phi'^{(R'_{\chi}/(R'_{\chi}-Z))}_{s}(X)}
\Gamma(X,\gr_{R'_{\chi}/(R'_{\chi}-Z)}j_{*}\Omega_{U}^{1}\otimes_{\dvr_{Z}}\dvr_{Z^{1/p}})
=\Gamma(Z^{1/p},\Omega^{1}(R'_{\chi})\otimes_{\dvr_{X}}\dvr_{Z^{1/p}}). \notag
\end{align}
\end{defn}

\section{Abbes-Saito's ramification theory and Witt vectors}
\label{secasram}
\subsection{Abbes-Saito's ramification theory}
\label{ssasram}
We briefly recall Abbes-Saito's non-logarithmic ramification theory
(\cite[Section 2, Subsection 3.1]{sa1}).

\begin{defn}[{\cite[Definition 1.12]{sa1}}]
Let $P$ be a scheme.
Let $D$ be a Cartier divisor on $P$ and $X$ a closed subscheme of $P$.
We define the \textit{dilatation} $P^{(D\cdot X)}$ of $P$ with respect to $(D,X)$
to be the complement of the proper transform of $D$ in the blow-up of $X$ along $D\cap X$.
\end{defn}

Let $X$ be a smooth separated scheme over a perfect field $k$ of characteristic $p>0$.
Let $D$ be a divisor on $X$ with simple normal crossings and 
$\{D_{i}\}_{i\in I}$ the irreducible components of $D$.
We put $U=X-D$.
Let $R=\sum_{i\in I}r_{i}D_{i}$ be a linear combination of integral coefficients
$r_{i}\ge 1$ for every $i\in I$.
Let $Z$ be the support of $R-D$.

We put $P=X\times_{k}X$.
Let $\Delta\colon X\rightarrow P$ be the diagonal and $\pr_{i}\colon P\rightarrow X$
the $i$-th projection for $i=1,2$.
We identify $D\subset X$ with closed subschemes of $P$ by the diagonal.
We put $P^{(D)}=\bigcap_{i=1}^{2}P^{(\pr_{i}^{*}D\cdot X)}$, where the intersection is taken in
the blow-up of $P$ along $D\subset P$.

Let $D^{(D)}_{i}$ be the inverse image of $D_{i}$ by $P^{(D)}\rightarrow P$.
Then $D^{(D)}=\sum_{i\in I}D_{i}^{(D)}$ is a divisor on $P^{(D)}$ with simple normal crossings.
The diagonal $\Delta$ is canonically lifted to the closed immersion $X\rightarrow P^{(D)}$
and we identify $X$ with a closed subscheme of $P^{(D)}$ by the lift.
We define $P^{(R)}$ to be the dilatation of $P^{(D)}$ 
with respect to $(\sum_{i\in I}(r_{i}-1)D^{(D)}_{i},X)$.
Let $T^{(R)}\subset D^{(R)}$ be the inverse image of $Z\subset D$ by $P^{(R)}\rightarrow P$.
Then the complement $P^{(R)}-D^{(R)}$ is $U\times_{k}U$ (\cite[Lemma 2.4.4]{sa1})
and $T^{(R)}$ is $TX(-R)\times _{X}Z$ (\cite[Corollary 2.9]{sa1}), 
where $TX=\Spec S^{\bullet}\Omega_{X}^{1}$ denotes the tangent bundle of $X$.

Let $G$ be a finite group and $V\rightarrow U$ a $G$-torsor.
We consider the open immersion $U\times_{k}U=P^{(R)}-D^{(R)}\rightarrow P^{(R)}$.
The quotient $(V\times_{k}V)/\Delta G$ of $V\times_{k}V$ by the diagonal action of $G$
is finite \'{e}tale over $U\times_{k}U$.
Let $Q^{(R)}$ be the normalization of $P^{(R)}$ in the finite \'{e}tale covering
$(V\times_{k}V)/\Delta G\rightarrow U\times_{k}U$.
Then the canonical lift $X\rightarrow P^{(R)}$ of the diagonal is canonically lifted to
$X\rightarrow Q^{(R)}$.

\begin{defn}[{\cite[Definition 2.12]{sa1}}]
Let $V\rightarrow U$ be a $G$-torsor for a finite group $G$
and $R=\sum_{i\in I}r_{i}D_{i}$ a linear combination of integral coefficients
$r_{i}\ge 1$ for every $i\in I$.
\begin{enumerate}
\item We say that the ramification of $V$ over $U$ at a point $x$ on $D$ is \textit{bounded by} $R+$
if the finite morphism $Q^{(R)}\rightarrow P^{(R)}$ is \'{e}tale on a neighborhood of
the image of $x$ by the lift $X\rightarrow Q^{(R)}$.
\item We say that the ramification of $V$ over $U$ along $D$ is \textit{bounded by} $R+$
if the finite morphism $Q^{(R)}\rightarrow P^{(R)}$ is \'{e}tale on a neighborhood of
the image of the lift $X\rightarrow Q^{(R)}$.
\end{enumerate}
\end{defn}

\begin{lem}
\label{lempurity}
Let $V\rightarrow U$ be a $G$-torsor for a finite group $G$
and $R=\sum_{i\in I}r_{i}D_{i}$ a linear combination of integral coefficients
$r_{i}\ge 1$ for every $i\in I$.
Let $\mathfrak{p}_{i}$ be the generic point of $D_{i}$ for $i\in I$.
Then the following are equivalent:
\begin{enumerate}
\item The ramification of $V$ over $U$ at $\mathfrak{p}_{i}$
is bounded by $R+$ for every $i\in I$.
\item The ramification of $V$ over $U$ along $D$ is bounded by $R+$.
\end{enumerate}
\end{lem}

\begin{proof}
Since $Q^{(R)}\rightarrow P^{(R)}$ is an isomorphism outside of the inverse image of $D$,
the assertion follows by the purity of Zariski-Nagata.
\end{proof}

In \cite{sa1}, the notion of the bound of ramification of $V$ over $U$ is defined for $R=\sum_{i\in I}r_{i}D_{i}$
of rational coefficients $r_{i}\ge 1$.
The next proposition relates the ramification of $G$-torsor to the ramification of local field.

\begin{prop}[{\cite[Proposition 2.27]{sa1}}]
\label{asrameq}
Assume that $D$ is irreducible.
Let $K$ be the local field at the generic point $\mathfrak{p}$ of $D$.
Let $\{G_{K}^{r}\}_{r\in \mathbf{Q}_{>0}}$ be the ramification filtration of
the absolute Galois group $G_{K}$ of $K$ (\cite[Definition 3.4]{as1}).
Let $r\ge 1$ be a rational number and
let $G_{K}^{r+}=\displaystyle{\overline{\bigcup_{s>r}G_{K}^{s}}}$ 
denote the closure of the union of $G_{K}^{s}$ for $s>r$.
For a $G$-torsor $V\rightarrow U$ for a finite group $G$, the following are equivalent:
\begin{enumerate}
\item The ramification of $V$ over $U$ at $\mathfrak{p}$ is bounded by $rD+$.
\item $G_{K}^{r+}$ acts trivially on the finite \'{e}tale $K$-algebra 
$L=\Gamma(V\times_{U}K,\dvr_{V\times_{U}K})$.
\end{enumerate}
\end{prop}

We note that the filtration $\{G_{K}^{r}\}_{r\in \BQ_{>0}}$ is decreasing.

We recall the characteristic form defined in \cite[Subsection 2.4]{sa1}.
Let $W^{(R)}$ be the largest open subscheme of $Q^{(R)}$ \'{e}tale over $P^{(R)}$.
We define a scheme $E^{(R)}$ over $T^{(R)}$ to be 
the fiber product $T^{(R)}\times_{P^{(R)}}W^{(R)}$.
Then there is a unique open sub group scheme $E^{(R)0}$ of a smooth group scheme $E^{(R)}$
over $Z$ such that for every $x\in Z$ the fiber $E^{(R)0}\times_{Z}x$ is 
the connected component of $E^{(R)}\times_{Z}x$ containing the unit section
(\cite[Proposition 2.16]{sa1}). 
Further $E^{(R)0}$ is \'{e}tale over $T^{(R)}$. 

Assume that the ramification of $V$ over $U$ along $D$ is bounded by $R+$.
Then, we say that the ramification of $V$ over $U$ along $D$ is \textit{non-degenerate}
at the multiplicity $R$
if the \'{e}tale morphism $E^{(R)0}\rightarrow T^{(R)}$ is finite.
We note that this condition is satisfied 
if we remove a sufficiently large closed subscheme of $X$ of codimension $\ge 2$.
Assume that the ramification of $V$ over $U$ along $D$ is non-degenerate
at the multiplicity $R$.
Then the exact sequence $0\rightarrow \tilde{G}^{(R)}\rightarrow E^{(R)0}
\rightarrow T^{(R)}\rightarrow 0$ defines
a closed immersion $\tilde{G}^{(R) \vee} \rightarrow T^{(R) \vee}$ of commutative group schemes
to the dual vector bundle defined over $Z^{1/p^{n}}$, where $n\ge 0$ is an integer.

\begin{defn}[{\cite[Definition 2.19]{sa1}}]
Let $V\rightarrow U$ be a $G$-torsor for a finite group $G$.
Assume that the ramification of $V$ over $U$ along $D$ is bounded by $R+$
and non-degenerate at the multiplicity $R$.
We define the \textit{characteristic form} $\Char_{R}(V/U)$ to be the morphism 
$\tilde{G}^{(R) \vee} \rightarrow T^{(R) \vee}=(T^{*}X\times_{X}Z)(R)$ over $Z^{1/p^{n}}$
for a sufficiently large integer $n\ge 0$.
\end{defn}

\begin{prop}[{cf.\ \cite[Corollary 2.28.2]{sa1}}]
\label{propkcf}
Let the notation be as in Proposition \ref{asrameq}.
Let $\dvr_{K}$ be the valuation ring of $K$ and 
$F_{K}$ the residue field of $K$.
We put $N^{(r)}=\mathfrak{m}_{\bar{K}}^{r}/\mathfrak{m}_{\bar{K}}^{r+}$,
where $\mathfrak{m}_{\bar{K}}^{r}=\{a\in \bar{K}\; |\;\ord_{K}(a)\ge r\}$
and $\mathfrak{m}_{\bar{K}}^{r+}=\{a\in \bar{K}\; |\;\ord_{K}(a)> r\}$.
Let $r\ge 1$ be a rational number.
Assume that the ramification of $V$ over $U$ along $D$ is bounded by $R+$
and non-degenerate at the multiplicity $rD$.
Then the following are equivalent:
\begin{enumerate}
\item The characteristic form $\Char_{rD}(V/U)$ defines the non-zero mapping
by taking the stalk at the generic point of $D$.
\item $G_{K}^{r+}$ acts non-trivially on $L$.
\end{enumerate}
\end{prop}

\begin{proof}
The assertion follows by \cite[Corollary 2.28.2]{sa1} and its proof.
\end{proof}

\subsection{Valuation of Witt vectors}

We keep the notation in Subsection \ref{ssasram}.
In this subsection, we assume that $X$ is an smooth affine scheme $\Spec A$ over $k$
and that $D$ is an irreducible divisor  defined by $\pi \in A$. 
We put $U=\Spec B$ and $R=rD$, where $r\ge 1$ is an integer.

Let $J\subset A$ be the kernel of the multiplication $A\otimes_{k}A\rightarrow A$.
Following the construction of $P^{(R)}$ recalled in the previous section, we have
\begin{equation}
P^{(R)}=\Spec (A\otimes_{k}A)[J/(\pi^{r}\otimes 1),
((1\otimes \pi)/(\pi\otimes 1))^{-1}]. \notag
\end{equation}
The divisor $D^{(R)}$ is defined by $t_{1}\otimes 1$. 

We put $P^{(R)}=\Spec A^{(r)}$. 
Let $\hat{A}$ denote the completion of the local ring $\dvr_{X,\mathfrak{p}}$ at the generic point 
$\mathfrak{p}$ of $D$
and $\hat{A}^{(r)}$ the completion of the local ring $\dvr_{P^{(R)},\mathfrak{q}}$
at the generic point $\mathfrak{q}$ of $D^{(R)}$ respectively.
Let $u\colon \hat{A}\rightarrow \hat{A}^{(r)}$ and
$v\colon \hat{A}\rightarrow \hat{A}^{(r)}$ be the morphisms
induced by the first and second projections $P\rightarrow X$ respectively. 
We put $K=\Frac \hat{A}$ and $L^{(r)}=\Frac \hat{A}^{(r)}$.

\begin{lem}
\label{lemvalvua}
Let $F_{K}$ be the residue field of $K$.
Let $a=a'\pi^{n}\in K$ be an element, where $n$ is an integer and $a'\in \hat{A}^{\times}$ is a unit.
Let $r\ge 1$ be an integer.
\begin{enumerate}
\item If $n=0$ and if $r=1$, then we have $\ord_{L^{(r)}}(v(a)/u(a))=0$.
\item If $n\notin p\BZ$ or $r=1$, then $\ord_{L^{(r)}}(v(a)/u(a)-1)=r-1$.
\item If $n\in p\BZ$ and if $r>1$, then $\ord_{L^{(r)}}(v(a)/u(a)-1)\ge r$.
Further if $a'$ is not a $p$-power in $F_{K}$, the equality holds.
\end{enumerate}
\end{lem}

\begin{proof}
We put $w=(v(\pi)-u(\pi))/u(\pi)^{r}$ and $w'=(v(a')-u(a'))/u(\pi)^{r}$.
Then we have $v(\pi)/u(\pi)=1+u(\pi)^{r-1}w$ and $v(a')/u(a')=1+u(a')^{-1}u(\pi)^{r}w'$.
Hence we have 
\begin{equation}
\label{vauamo}
v(a)/u(a)-1=\begin{cases}
(1+u(a')^{-1}u(\pi)^{r}w')(1+u(\pi)^{r-1}w)^{n}-1 & (n\ge 0) \\
(1+u(\pi)^{r-1}w)^{n}\left(\left(1+u(a')^{-1}u(\pi)^{r}w' \right)-\left(1+u(\pi)^{r-1}w\right)^{-n}\right)
& (n<0).
\end{cases} 
\end{equation}

Suppose that $n=0$ and $r=1$.
Then we have $v(a)/u(a)=1+u(a')^{-1}u(\pi)w'$.
Since $u(\pi)=\pi \otimes 1$ is a uniformizer of $\hat{A}^{(r)}$,
the assertion (i) holds.

Suppose that $n\notin p\BZ$.
Then we have $\ord_{L^{(r)}}(v(a)/u(a)-1)=r-1$.

Assume that $n\in p\BZ$.
Suppose that $n= 0$.
Then we have $\ord_{L^{(r)}}(v(a)/u(a)-1)\ge r$,
and the equality holds if $w'$ is a unit in $\hat{A}^{(r)}$.

Suppose that $n\neq 0$.
We put $n=p^{s'}n'$, where $s'=\ord_{p}(n)\ge 1$.
Then we have $\ord_{L^{(r)}}(v(a)/u(a)-1)\ge \min\{r,p^{s'}(r-1)\}$.
If $r=1$, then we have $r>p^{s'}(r-1)=0=r-1$.
Since $w\in \hat{A}^{(r)\times}$ is a unit, the assertion follows if $r=1$.

If $r>1$,
then $p^{s'}(r-1)\ge r$.
Further the equality holds only if $(p,r,s')=(2,2,1)$.
Hence we have $\ord_{L^{(r)}}(v(a)/u(a)-1)\ge r$.
Further, if $(p,r,s')\neq(2,2,1)$ and 
if $w'$ is a unit in $\hat{A}^{(r)}$, the equality holds.
If $(p,r,s')=(2,2,1)$,
then we have $\ord_{L^{(r)}}(v(a)/u(a)-1)=r$ if and only if $u(a')^{-1}w'\neq n'w^{p}$.

Assume that $a$ is not a $p$-power in $F_{K}$.
Then the elements $\pi$ and $a'$ are $p$-independent over $K^{p}$.
Hence the images in $\hat{A}^{(r)}/u(\pi)\hat{A}^{(r)}$ of $w$ and $w'$
form a part of a basis of the $F_{K}$-vector space $\pi^{-r}\Omega_{A}^{1}\otimes_{A}F_{K}$,
since $T^{(R)}=TX(-R)\times_{X}D$.
Hence $w'$ is a unit in $\hat{A}^{(r)}$ and $u(a')^{-1}w'\neq n'w^{p}$.
Thus the assertions (ii) and (iii) follow.
%
\end{proof}

We put $\BZ[T,S]_{d}=\mathbf{Z}[T_{d}, \ldots, T_{s-1},S_{d},\ldots, S_{s-1}]$
for an integer $d$ such that $0\le d \le s-1$.
We define polynomials $Q_{d}(T,S)\in \BZ[T,S]_{d}[1/p]$ for $0\le d \le s-1$ inductively by the relation
\begin{equation}
\label{defofq}
\sum_{i=d}^{s-1}p^{s-1-i}(T_{i}(1+S_{i}))^{p^{i-d}}=\sum_{i=d}^{s-1}p^{s-1-i}T_{i}^{p^{i-d}}+\sum_{i=d}^{s-1}p^{s-1-i}Q_{i}^{p^{i-d}}. 
\end{equation}
It is well-known in the theory of Witt vectors that $Q_{d}$ is an element of $\BZ[T,S]_{d}$.

For elements $x=(x_{s-1}, \ldots x_{0})$ and $y=(y_{s-1}, \ldots, y_{0})$ of $W_{s}(A)$ for a ring $A$, we put $x^{\prime}=(x_{s-1}^{\prime},\ldots ,x_{0}^{\prime})$, where $x_{i}^{\prime}=x_{i}(1+y_{i})$ for $i=0, \ldots ,s-1$.
Then we have 
\begin{equation}
\label{xminusxp}
x-x^{\prime}=(Q_{s-1}(x,y),Q_{s-2}(x,y), \ldots, Q_{0}(x,y)). 
\end{equation} 

\begin{lem}[cf.\ {\cite[Lemma 12.2]{as3}}]
\label{aswitt}
Let the notation be as above.
\begin{enumerate}
\item $Q_{d}(T,S)$ belongs to the ideal of $\BZ[T,S]_{d}$ 
generated by $(S_{i})_{d\le i\le s-1}$ for $d=0,\ldots ,s-1$. 
\item $Q_{d}(T,S)-\sum_{i=d}^{s-1}T_{i}^{p^{i-d}}S_{i}$ belongs to the ideal 
of $\BZ[T,S]_{d}$ 
generated by $(S_{i}S_{j})_{d\le i,j\le s-1}$ for $d=0,\ldots ,s-1$.
\item If we replace $T_{i}$ by $T_{i}^{p^{s-1-i}}$ in $Q_{d}(T,S)$, the polynomial $Q_{d}(T,S)$ is
homogeneous of degree $p^{s-1-d}$ as a polynomial of multi-value $T$
for $0\le d \le s-1$. 
\end{enumerate}
\end{lem}

\begin{proof}
The assertions (i) and (ii) are the same as (i) and (ii) in \cite[Lemma 12.2]{as3} respectively.

We prove (iii) by the induction on $d$.
If $d=s-1$, we have $Q_{s-1}=T_{s-1}S_{s-1}$.
Hence the assertion follows.

If $d<s-1$, we have
\begin{equation}
Q_{d}=p^{d-s+1}\left( \sum_{i=d}^{s-1}p^{s-1-i}T_{i}^{p^{i-d}}\left( (1+S_{i})^{p^{i-d}}-1\right)
-\sum_{i=d+1}^{s-1}p^{s-1-i}Q_{i}^{p^{i-d}}\right). \notag
\end{equation}
By the induction hypothesis, the polynomial $Q_{i}(T,S)$ is homogeneous of degree $p^{s-1-i}$
for $T$ for $d+1\le i \le s-1$
with $T_{j}$ replaced by $T_{j}^{p^{s-1-j}}$ for $i\le j \le s-1$.
Hence $Q_{i}(T,S)^{p^{i-d}}$ is homogeneous of degree $p^{s-1-d}$ for $T$ for $d+1\le i\le s-1$ 
with the same replacement of $T_{j}$ for $i\le j \le s-1$.
Hence the assertion follows.
\end{proof}

\begin{lem}
\label{lemvalq}
Let $a=(a_{s-1},\ldots,a_{0})$ be an element of $W_{s}(K)$ and
put $b=(b_{s-1},\ldots,b_{0})\in W_{s}(L^{(r)})$, where
$b_{i}=v(a_{i})/u(a_{i})-1$ if $a_{i}\neq 0$ and $b_{i}=0$ if $a_{i}=0$ 
for $0\le i\le s-1$.
Let $m\ge 1$ be an integer and assume that $a\in \fil_{m}W_{s}(K)$.
Let $r\ge 1$ be an integer.
\begin{enumerate}
\item If $(m,r)=(1,1)$, then $p^{i}\ord_{L^{(r)}}(Q_{d}(u(a),b))\ge -m+1$ for every $0\le d\le s-1$.
\item If $r>1$, then
$p^{i}\ord_{L^{(r)}}(Q_{d}(u(a),b))>-m+r$ for every $0<d\le s-1$, 
and $\ord_{L^{(r)}}(Q_{0}(u(a),b))\ge -m+r$.
\end{enumerate}
\end{lem}

\begin{proof}
We put $s'=\min\{\ord_{p}(m),s\}$.
Let $a'=(a'_{s-1},\ldots,a_{0}')$ be an element of $W_{s}(K)$ such that
$a_{i}'=0$ if $p^{i}\ord_{K}(a_{i})=-m$ and $a_{i}'=a_{i}$ if $p^{i}\ord_{K}(a_{i})\ge -(m-1)$ for
$0\le i \le s-1$.
We note that if $s'\le i \le s-1$ then $a_{i}'=a_{i}$ by (\ref{filpwdef}).
Let $a''=(a_{s'-1}'',\ldots,a_{0}'')$ be an element of $W_{s'}(K)$ such that
$a_{i}''=0$ if $p^{i}\ord_{K}(a_{i})\ge -(m-1)$ and $a_{i}''=a_{i}$ if $p^{i}\ord_{K}(a_{i})=-m$
for $0\le i \le s'-1$.
Then we have $a=a'+V^{s-s'}(a'')$.
Let $b'\in W_{s}(L^{(r)})$ and $b''\in W_{s'}(L^{(r)})$ be the elements defined from
$a'$ and $a''$ respectively similarly as $b$ defined from $a$. 
Since we have $Q(u(a),b)=(Q_{s-1}(u(a),b),\ldots,Q_{0}(u(a),b))=v(a)-u(a)$ and 
similarly for $a'$ and $a''$ by (\ref{xminusxp}),
we have $Q(u(a),b)=Q(u(a'),b')+V^{s-s'}(Q(u(a''),b''))$.
Since $\fillog_{n}W_{s}(L^{(r)})$ is a submodule of $W_{s}(L^{(r)})$ for $n\in \BZ$,
the assertion follows for $a$ if the assertions follows for $a'$ and $a''$.
Hence we prove the assertions for $a'$ and $a''$.

By the definitions of $a'$ and $a''$,
we have $\ord_{L^{(r)}}(u(a_{i}'))\ge -(m-1)/p^{i}$ for $0\le i \le s-1$ 
and $\ord_{L^{(r)}}(u(a''_{i}))\ge -m/p^{i}$
for $0\le i\le s'-1$.
If $r>1$, then we have $\ord_{L^{(r)}}(b_{i}')\ge r-1$ for $0\le i\le s-1$
and $\ord_{L^{(r)}}(b_{i}'')\ge r$ for $0\le i \le s'-1$ by Lemma \ref{lemvalvua} (ii) and (iii).
If $(m,r)=(1,1)$, then we have $s'=0$ and $\ord_{L^{(r)}}(b_{i}')\ge r-1$ for $0\le i\le s-1$ 
by Lemma \ref{lemvalvua} (ii).
Hence, by Lemma \ref{aswitt} (i) and (iii), we have
\begin{equation}
\label{eqvalap}
p^{d}\ord_{L^{(r)}}(Q_{d}(u(a'),b'))\ge -(m-1)+p^{d}(r-1)\ge -m+r. 					
\end{equation}
Further we have
\begin{equation}
\label{eqvalapp}
p^{d}\ord_{L^{(r)}}(Q_{d}(u(a''),b''))\ge -m+p^{d}r\ge -m+r. 					
\end{equation}
If $r> 1$, then the right equality in (\ref{eqvalap}) holds only if $d=0$
and so in (\ref{eqvalapp}).
Hence the assertions follow.
\end{proof}

\begin{lem}
\label{valqz}
Let the notation be as in Lemma \ref{lemvalq}.
Let $m\ge 2$ be an integer and assume that $a\in\fil_{m}W_{s}(K)$.
Then we have
$\ord_{L^{(m)}}(Q_{0}(u(a),b)-\sum_{i=0}^{s-1}u(a_{i})^{p^{i}}
b_{i})>0$.
\end{lem}

\begin{proof}
We put $s'=\min\{\ord_{p}(m),s\}$.
Let $a'=(a'_{s-1},\ldots, a_{0}')$ and $a''=(a_{s'-1}'',\ldots,a_{0}'')$
be as in the proof of Lemma \ref{lemvalq}.
We have $a=a'+V^{s-s'}(a'')$.
Let $b'\in W_{s}(L^{(m)})$ and $b''\in W_{s'}(L^{(m)})$ be the elements defined from
$a'$ and $a''$ respectively similarly as $b$ defined from $a$. 
Since $Q(u(a),b)=Q(u(a'),b')+V^{s-s'}Q(u(a''),b'')$ as in the proof of Lemma \ref{lemvalq}
and $\sum_{i=0}^{s-1}u(a_{i})^{p^{i}}b_{i}=\sum_{i=0}^{s-1}u(a'_{i})^{p^{i}}b_{i}'
+\sum_{i=0}^{s'-1}u(a''_{i})^{p^{i}}b''_{i}$,
it is sufficient to prove the assertion for $a'$ and $a''$.

As in the proof of Lemma \ref{lemvalq}, we have
$\ord_{L^{(m)}}(u(a'_{i}))\ge -(m-1)/p^{i}$ for $0\le i \le s-1$ and $\ord_{L^{(m)}}(u(a''_{i}))\ge -m/p^{i}$
for $0\le i \le s'-1$.
Further we have $\ord_{L^{(m)}}(b_{i}')\ge m-1$ for $0\le i\le s-1$
and $\ord_{L^{(m)}}(b_{i}'')\ge m$ for $0\le i \le s'-1$.
Hence, by Lemma \ref{aswitt} (ii) and (iii), we have
\begin{equation}
\ord_{L^{(m)}}(Q_{0}(u(a'),b')-\sum_{i=0}^{s-1}u(a_{i}')^{p^{i}}b_{i}')
\ge -(m-1)+2(m-1)=m-1>0. \notag
\end{equation}
Further we have
\begin{equation}
\ord_{L^{(m)}}(Q_{0}(u(a''),b'')-\sum_{i=0}^{s'-1}u(a_{i}'')^{p^{i}}b_{i}'')
\ge -m+2m=m>0. \notag
\end{equation}
Hence the assertion follows.
\end{proof}

\subsection{Calculation of characteristic forms}

Let $X$ be a smooth separated scheme over a perfect field $k$ of characteristic $p>0$.
Let $D$ be a divisor on $X$ with simple normal crossings and 
$\{D_{i}\}_{i\in I}$ the irreducible components of $D$.
We put $U=X-D$ and let $j\colon U\rightarrow X$ denote the canonical open immersion.
Let $K_{i}$ be the local field at the generic point of $D_{i}$ for $i\in I$ and
let $\dvr_{K_{i}}$ be the valuation ring of $K_{i}$ for $i\in I$.

Let $\chi$ be an element of $H^{1}_{\et}(U,\BQ/\BZ)$.
In this subsection, we prove the equality of the characteristic form $\cform(\chi)$ of $\chi$
and the characteristic form $\Char_{R}(V/U)$ of the Galois torsor 
$V\rightarrow U$ corresponding to $\chi$.

Let $p_{i}\colon P^{(R)}\rightarrow X$ be the morphism 
induced by the $i$-th projection for $i=1,2$.
Let $u\colon p_{1}^{-1}\dvr_{X}\rightarrow \dvr_{P^{(R)}}$
and $v\colon p_{2}^{-1}\dvr_{X}\rightarrow \dvr_{P^{(R)}}$
be the canonical morphisms of sheaves on $P^{(R)}$ by abuse of notation.
Let $L_{i}^{(R)}$ be the fractional field of the completion of the local ring 
$\dvr_{P^{(R)},\mathfrak{q}_{i}}$, where $R=\sum_{i\in I}r_{i}D_{i}$ is a linear combination
of integer coefficients $r_{i}\ge 1$ for every $i\in I$ and
$\mathfrak{q}_{i}$ is the generic point of the pull-back
$D_{i}^{(R)}$ of $D_{i}^{(D)}$ by $P^{(R)}\rightarrow P^{(D)}$.
If $D=D_{1}$ is irreducible, then we simply write $L^{(r_{1})}$ for $L_{1}^{(R)}$
as in the previous section.

We first consider the tamely ramified case.

\begin{lem}
\label{lemkum}
Assume that the order $n$ of $\chi$ is prime to $p$.
Take an inclusion $\mu_{n}\rightarrow \BQ/\BZ$ so that
the image of $\chi$ is contained in $\mu_{n}\subset \BQ/\BZ$.
We put $G=\mu_{n}$.
Let $V\rightarrow U$ be the $G$-torsor corresponding to $\chi$.
Let $R=\sum_{i\in I}r_{i}D_{i}$ be a linear combination of integral
coefficients $r_{i}\ge 1$ for every $i\in I$.
\begin{enumerate}
\item The ramification of $V$ over $U$ along $D$ is bounded by $D+$.
\item The characteristic form $\Char_{R}(V/U)$ is the zero mapping.
\end{enumerate}
\end{lem}

\begin{proof}
(i) By Lemma \ref{lempurity}, 
we may assume that $D=D_{1}$ is irreducible.
Since the assertion is local, we may assume that $X=\Spec A$ is
affine and $D$ is defined by an element of $A$.
Since $\ord_{L^{(r_{1})}}(v(a)/u(a))= 0$ for every unit 
$a\in \dvr_{K_{1}}^{\times}$ by Lemma \ref{lemvalvua} (i),
the assertion follows.

(ii) Let $Z$ be the support of $R-D$.
By (i) and Proposition \ref{asrameq}, the ramification group $G_{K_{i}}^{r_{i}+}$
acts trivially on $L_{i}=\Gamma(V\times_{U}K_{i},\dvr_{V\times_{U}K_{i}})$ for $D_{i}$ contained
in $Z$.
By Proposition \ref{propkcf}, the stalk of the characteristic form $\Char_{R}(V/U)$ 
at the generic point of $D_{i}$ defines the zero mapping for $D_{i}$ contained in $Z$. 
Hence the assertion follows.
\end{proof}

By Lemma \ref{lemkum}, the bound of 
the ramification of the Galois torsor $V\rightarrow U$ corresponding to $\chi$
and its characteristic form $\Char_{R}(V/U)$ does not depend on the prime-to-$p$-part of $\chi$, 
that is, they are dependent only on the $p$-part of $\chi$.

\begin{prop}
\label{propeqcf}
Assume that the order of $\chi$ is $p^{s}$ and
take an inclusion $\BZ/p^{s}\BZ\rightarrow \BQ/\BZ$ 
such that the image of $\chi$ is contained in $\BZ/p^{s}\BZ$.
We put $G=\BZ/p^{s}\BZ$.
Let $V\rightarrow U$ be a $G$-torsor corresponding to $\chi$.
\begin{enumerate}
\item The ramification of $V$ over $U$ along $D$ is bounded by $R_{\chi}'+$,
where $R_{\chi}'$ is the total dimension divisor of $\chi$ (Definition \ref{defdtdiv}).
\item Assume that $R_{\chi}'\neq D$ and put $Z=\Supp (R'_{\chi}-D)$.
Then the scheme $E^{(R_{\chi}')}\rightarrow T^{(R_{\chi}')}=TX(-R_{\chi}')\times_{X}Z$ is
defined by the Artin-Schreier equation $t^{p}-t=\cform(\chi)$.
\end{enumerate}
\end{prop}

\begin{proof}
We put $m_{i}=\dt(\chi|_{K_{i}})$ for $i\in I$.
Let $a=(a_{s-1},\ldots,a_{0})\in \fil_{R_{\chi}'}j_{*}W_{s}(\dvr_{U})$ be an element
whose image by $\delta_{s}$ (\ref{deltsshf}) is $\chi$.
Then $V\times_{k}V/\Delta G\rightarrow U\times_{k}U$ is the $G$-torsor
defined by the Artin-Schreier-Witt equation $(F-1)(t)=v(a)-u(a)$.

(i) By Lemma \ref{lempurity}, 
we may assume that $D$ is irreducible.
Since the assertion is local, we may assume that $X=\Spec A$ is
affine and that $D$ is defined by an element of $A$.
By (\ref{xminusxp}) and Lemma \ref{lemvalq}, the difference $v(a)-u(a)$ is a regular function on $P^{(R_{\chi}')}$.
Hence the assertion follows.

(ii) By (i), (\ref{xminusxp}), Lemma \ref{lemvalq} (ii), and Lemma \ref{valqz}, 
the scheme $E^{(R'_{\chi})}\rightarrow T^{(R'_{\chi})}$
is the $G$-torsor defined by the Artin-Schreier equation $t^{p}-t=\sum_{j=0}^{s-1}u(a_{j})^{p^{j}-1}
(v(a_{j})-u(a_{j}))$.
We put $n_{ij}=\ord_{K_{i}}(a_{j})$ for $i\in I$ and $0\le j \le s-1$.
As calculating in the proof of Lemma \ref{lemvalvua},
we have the following on a neighborhood of the generic point of
$D_{i}^{(R_{\chi}')}$ for $i\in I$ such that $m_{i}>1$:
\begin{enumerate}
\renewcommand{\labelenumi}{(\alph{enumi})}
\item If $n_{ij}\notin p\BZ$, we have 
$u(a_{j})^{p^{j}-1}(v(a_{j})-u(a_{j}))=
n_{ij}u(a_{j})^{p^{j}}u(t_{i})^{m_{i}-1}w_{i}$;
\item If $n_{ij}\in p\BZ$ and if $(p,m_{i},\ord_{p}(n_{ij}))\neq (2,2,1)$,
we have $u(a_{j})^{p^{j}-1}(v(a_{j})-u(a_{j}))=
u(a_{j})^{p^{j}}u(a_{j}')^{-1}u(t_{i})^{m_{i}}w_{ij}'$;
\item If $(p,m_{i},\ord_{p}(n_{ij}))= (2,2,1)$,
we have $u(a_{j})^{p^{j}-1}(v(a_{j})-u(a_{j}))=u(a_{j})^{p^{j}}
(u(a_{j}')^{-1}u(t_{i})^{2}w_{ij}'+(n_{ij}/2)u(t_{i})^{2}w_{i}^{2})$,
\end{enumerate}
where $t_{i}$ is a local equation of $D_{i}$, $a_{j}'=a_{j}/t_{i}^{n_{ij}}$, 
$w_{i}=(v(t_{i})-u(t_{i}))/u(t_{i})^{m_{i}}$, and
$w_{ij}'=(v(a_{j}')-u(a_{j}'))/u(t_{i})^{m_{i}}$
for every $j=0,\ldots,s-1$.
Since $a\in \fil_{R_{\chi}'}j_{*}W_{s}(\dvr_{U})$,
we have $p^{j}\ord_{L_{i}^{(m_{i})}}(a_{j})\ge -(m_{i}-1)$ if $n_{ij}\notin p\BZ$
and $p^{j}\ord_{L_{i}^{(m_{i})}}(a_{j})\ge -m_{i}$ if $n_{ij}\in p\BZ$.
If $(p,m_{i},\ord_{p}(n_{ij}),p^{j}n_{ij})=(2,2,1,-2)$, we have $(p,j,n_{ij})=(2,0,-2)$. 
Hence the assertion follows by identifying $w_{i}$ and $w_{ij}'$ with
$dt_{i}/t_{i}^{m_{i}}$ and $da_{j}'/t_{i}^{m_{i}}$ respectively.
\end{proof}

\begin{cor} 
\label{corcf}
Let $V\rightarrow U$ be the Galois torsor corresponding to $\chi$.
Assume that the ramification of $V$ over $U$ along $D$ is non-degenerate
at the multiplicity $R_{\chi}'$.
\begin{enumerate}
\item The image of the generator 
$1\in \tilde{G}^{(R_{\chi}')\vee}=\mathbf{F}_{p}$
by $\Char_{R_{\chi}'}(V/U)$
is equal to $\cform(\chi)$.
\item Assume that $D=D_{1}$ is irreducible and 
that $\dt(\chi|_{K_{1}})>1$.
Then the ramification of $V$ over $U$ at the generic point of $D$
is not bounded by $rD+$ for any rational number $r$ such that
$1\le r<\dt(\chi|_{K_{1}})$.
\end{enumerate}
\end{cor}

\begin{proof}
(i) The assertion follows by Lemma \ref{lemkum} and Proposition \ref{propeqcf} (ii).

(ii) We put $K=K_{1}$.
Assume that $G_{K}^{r+}$ acts trivially 
on $L=\Gamma (V\times_{U}K,\dvr_{V\times_{U}K})$
for a rational number $r$ such that $1\le r< \dt(\chi|_{K})$.
Then, by (i) and Proposition \ref{propkcf}, 
the stalk $\cform(\chi|_{K})$ of $\cform(\chi)$ at the generic point of $D$ must be $0$.
However $\cform(\chi)$ is non-zero.
Hence the assertion follows by Proposition
\ref{asrameq}.
\end{proof}

\section{Equality of ramification filtrations}
\label{seceqfil}

Let $K$ be a complete discrete valuation field of characteristic $p>0$ and $F_{K}$ the residue field.
Let $G_{K}$ be the absolute Galois group of $K$.
We show that the abelianization of 
Abbes-Saito's filtration $\{G_{K}^{r}\}_{r \in \mathbf{Q}>0}$ (\cite[Definition 3.4]{as1})
is the same as 
$\{\fil_{m}H^{1}(K,\mathbf{Q}/\mathbf{Z})\}_{m\in \mathbf{Z}_{\ge 1}}$ (Definition \ref{defoffiltofh1})
by taking dual.
If $m>2$, then it has been proved by Abbes-Saito (\cite[Th\'{e}or\`{e}me 9.10]{as3}).

\begin{thm}
\label{thmeqfil}
Let $\chi$ be an element of $H^{1}(K,\mathbf{Q}/\mathbf{Z})$.
Let $m\ge 1$ be an integer.
Let $r$ be a rational number such that $m\leq r<m+1$.
If $F_{K}$ is finitely generated over a perfect subfield $k\subset F_{K}$, 
then the following are equivalent:
\begin{enumerate}
\item $\chi \in \fil_{m}H^{1}(K,\mathbf{Q}/\mathbf{Z})$.
\item $\chi(G_{K}^{m+})=0$.
\item $\chi(G_{K}^{r+})=0$.
\end{enumerate}
\end{thm}

\begin{proof}
Since $G_{K}^{1+}$ is a pro-$p$-subgroup of $G_{K}$ (\cite[Proposition 3.7.1]{as1}), 
we may assume that the order of $\chi$ is a power of $p$.
Let $p^{s}$ be the order of $\chi$ and put $G=\BZ/p^{s}\BZ$.
We take an inclusion $\BZ/p^{s}\BZ\rightarrow \BQ/\BZ$ such that
the image of $\chi$ is contained in $\BZ/p^{s}\BZ$.
As in \cite[6.1]{as3}, we take a smooth affine connected scheme $X$ over $k$ and 
a smooth irreducible divisor $D$ on $X$
such that the completion $\hat{\dvr}_{X,\mathfrak{p}}$ of the local ring $\dvr_{X,\mathfrak{p}}$ 
at the generic point $\mathfrak{p}$ of $D$ is isomorphic to $\dvr_{K}$.
By shrinking $X$ if necessary, we take a $G$-torsor $V\rightarrow U=X-D$
corresponding to a character of $\pi^{\ab}_{1}(U)$ whose restriction on $G_{K}$ is $\chi$.

By Proposition \ref{propeqcf} (i) and Corollary \ref{corcf} (ii), the ramification of $V$ over $U$
at the generic point of $D$ is bounded by $rD+$ for a rational number $r\ge 1$
if and only if $r\ge \dt(\chi)$.
Further, by Proposition \ref{asrameq}, the former condition is equivalent to that $G_{K}^{r+}$
acts trivially on $L=\Gamma(V\times_{U}K, \dvr_{V\times_{U}K})$.
Hence $\chi(G_{K}^{r+})=0$ if and only if $r\ge \dt(\chi)$.

Since the condition (i) holds if and only if $m\ge \dt(\chi)$, 
the equivalence of (i) and (ii) follows.
Since $m\le r$, the condition (ii) deduces the condition (iii).
Suppose that the condition (iii) holds.
Since $r\ge \dt(\chi)$, we have $m=[r]\ge \dt(\chi)$. 
Hence the condition (ii) holds.
\end{proof}

\begin{proof}[Proof of Theorem \ref{mainthm}]
We may identify $K$ with $F_{K}((\pi))$ by taking a uniformizer of $K$.
Let $K_{h}=\Frac (F_{K}[\pi]_{(\pi)}^{h})$ be the fractional field of the henselization of the 
localization $F_{K}[\pi]_{(\pi)}$ of $F_{K}[\pi]$ at the prime ideal $(\pi)$.
Since the completion of $K_{h}$ is $K$, the canonical morphisms
$G_{K}\rightarrow G_{K_{h}}$ and $H^{1}(K_{h},\BQ/\BZ)\rightarrow H^{1}(K,\BQ/\BZ)$ are
isomorphisms.

Let $k$ be a perfect subfield of $F_{K}$ and take
a separating transcendental basis $S$ of $F_{K}$ over $k$.
For a finite subextension $E$ of $F_{K}$ over $k(S')$, where $S'$ is a finite set of $S$,
let $K_{E,h}$ denote the fractional field of the henselization of the local ring $E[\pi]_{(\pi)}$.
Since $F_{K}=\varinjlim E$, we may identify $K_{h}$ with the inductive limit $\varinjlim K_{E,h}$ and
$H^{1}(K_{h},\BQ/\BZ)$ with $\varinjlim H^{1}(K_{E,h},\BQ/\BZ)$, where
$E$ runs through such subfields of $F_{K}$.

Let $\chi$ be an element of $H^{1}(K,\BQ/\BZ)$.
Take a subfield $E$ of $F_{K}$ such that $E$ is a subextension of $F_{K}$ over $k(S')$
for a finite subset $S'\subset S$ and that $\chi\in H^{1}(K_{E,h},\BQ/\BZ)$.
Let $K_{E}$ denote the completion of $K_{E,h}$.
We identify $H^{1}(K_{E},\BQ/\BZ)$ with $H^{1}(K_{E,h},\BQ/\BZ)$
and $\chi\in H^{1}(K_{E,h},\BQ/\BZ)$ with an element of $H^{1}(K_{E},\BQ/\BZ)$.
We prove that each condition in Theorem \ref{thmeqfil} holds for $K$ 
if and only if it holds for $K_{E}$.

Let $\dt_{K}(\chi)$ and $\dt_{K_{E}}(\chi)$ denote the total dimension of $\chi$
as an element of $H^{1}(K,\BQ/\BZ)$ and $H^{1}(K_{E},\BQ/\BZ)$ respectively.
We put $\dt_{K}(\chi)=n$ and $\dt_{K_{E}}(\chi)=n'$ and prove that $n=n'$.
Since $\fil_{m}W_{s}(K_{E})\subset \fil_{m}W_{s}(K)$ for every integer $m\ge 1$,
we have $\fil_{m}H_{1}(K_{E},\BQ/\BZ)\subset \fil_{m}H^{1}(K,\BQ/\BZ)$.
Hence we have $1\le n\le n'$, which proves that $n=1$ if $n'=1$.

Assume that $n'>1$.
Take an element $\bar{a}$ of $\gr_{n'}W_{s}(E(\pi))$
whose image in $\gr_{n'}H^{1}(K_{E},\BQ/\BZ)$ is $\chi$.
Let $\cform_{K}(\chi)$ and $\cform_{K_{E}}(\chi)$ denote the characteristic form of $\chi$
as an element of $H^{1}(K,\BQ/\BZ)$ and $H^{1}(K_{E},\BQ/\BZ)$ respectively.
Let $\dvr_{K}$ and $\dvr_{K_{E}}$ denote the valuation rings of $K$ and $K_{E}$
respectively.
Since $F_{K}$ is separable over $E$, we have an injection 
$\Omega^{1}_{E[\pi]_{(\pi)}}\rightarrow \Omega^{1}_{F_{K}[\pi]_{(\pi)}}$.
This injection induces the injection $\Omega_{\dvr_{K_{E}}}^{1}\rightarrow \Omega^{1}_{\dvr_{K}}$, 
and further the injection $\gr_{n'}\Omega_{K_{E}}^{1}\rightarrow \gr_{n'}\Omega^{1}_{K}$.
Hence the canonical morphism $\gr_{n'}\Omega_{K_{E}}^{1}\otimes_{F_{K}}F_{K}^{1/p}\rightarrow 
\gr_{n'}\Omega^{1}_{K}\otimes_{F_{K}}F_{K}^{1/p}$ is injective.
Since $\cform_{K_{E}}(\chi)\neq 0$, the image of $\cform_{K_{E}}(\chi)$ in 
$\gr_{n'}\Omega^{1}_{K}\otimes_{F_{K}}F_{K}^{1/p}$ is not $0$.
This implies that $\cform_{K}(\chi)$ is the image of $\cform_{K_{E}}(\chi)$ in $\gr_{n'}\Omega^{1}_{K}\otimes_{F_{K}}F_{K}^{1/p}$.
Hence we have $n=n'$.
Since the condition (i) in Theorem \ref{thmeqfil} holds for $K$
if and only if $m\ge n$ and similarly for $K_{E}$, 
the condition (i) in Theorem \ref{thmeqfil} for $K$ is equivalent to that for $K_{E}$.

Let $r\ge 1$ be a rational number.
Since $K$ is an extension of $K_{E}$ of ramification index $1$ and the extension
of residue fields is separable,
by applying \cite[Lemma 2.2]{as2},
the canonical morphism $G_{K}\rightarrow G_{K_{E}}$
induces the surjection $G_{K}^{s}\rightarrow G_{K_{E}}^{s}$ for every $s\in \BQ_{\ge 1}$.
Hence we have $\chi(G_{K}^{r+})=0$ if and only if $\chi(G_{K_{E}}^{r+})=0$,
which proves the assertions for conditions (ii) and (iii) in Theorem \ref{thmeqfil}. 
\end{proof}


\end{document}